\documentclass[numbers=enddot,12pt,final,onecolumn,notitlepage]{scrartcl}%
\usepackage[headsepline,footsepline,manualmark]{scrlayer-scrpage}
\usepackage[all,cmtip]{xy}
\usepackage{amssymb}
\usepackage{amsmath}
\usepackage{amsthm}
\usepackage{framed}
\usepackage{comment}
\usepackage{color}
\usepackage[breaklinks=True]{hyperref}
\usepackage[sc]{mathpazo}
\usepackage[T1]{fontenc}
\usepackage{tikz}
\usepackage{needspace}
\usepackage{tabls}
\providecommand{\U}[1]{\protect\rule{.1in}{.1in}}
\usetikzlibrary{arrows}
\newcounter{exer}
\theoremstyle{definition}
\newtheorem{theo}{Theorem}[section]
\newenvironment{theorem}[1][]
{\begin{theo}[#1]\begin{leftbar}}
{\end{leftbar}\end{theo}}
\newtheorem{lem}[theo]{Lemma}
\newenvironment{lemma}[1][]
{\begin{lem}[#1]\begin{leftbar}}
{\end{leftbar}\end{lem}}
\newtheorem{prop}[theo]{Proposition}
\newenvironment{proposition}[1][]
{\begin{prop}[#1]\begin{leftbar}}
{\end{leftbar}\end{prop}}
\newtheorem{defi}[theo]{Definition}
\newenvironment{definition}[1][]
{\begin{defi}[#1]\begin{leftbar}}
{\end{leftbar}\end{defi}}
\newtheorem{remk}[theo]{Remark}
\newenvironment{remark}[1][]
{\begin{remk}[#1]\begin{leftbar}}
{\end{leftbar}\end{remk}}
\newtheorem{coro}[theo]{Corollary}
\newenvironment{corollary}[1][]
{\begin{coro}[#1]\begin{leftbar}}
{\end{leftbar}\end{coro}}
\newtheorem{conv}[theo]{Convention}
\newenvironment{convention}[1][]
{\begin{conv}[#1]\begin{leftbar}}
{\end{leftbar}\end{conv}}
\newtheorem{quest}[theo]{Question}
\newenvironment{question}[1][]
{\begin{quest}[#1]\begin{leftbar}}
{\end{leftbar}\end{quest}}
\newtheorem{warn}[theo]{Warning}

\newtheorem{conj}[theo]{Conjecture}

\newtheorem{exam}[theo]{Example}
\newenvironment{example}[1][]
{\begin{exam}[#1]\begin{leftbar}}
{\end{leftbar}\end{exam}}
\newtheorem{exmp}[exer]{Exercise}

\newenvironment{statement}{\begin{quote}}{\end{quote}}

\let\sumnonlimits\sum
\let\prodnonlimits\prod
\let\cupnonlimits\bigcup
\let\capnonlimits\bigcap
\renewcommand{\sum}{\sumnonlimits\limits}
\renewcommand{\prod}{\prodnonlimits\limits}
\renewcommand{\bigcup}{\cupnonlimits\limits}
\renewcommand{\bigcap}{\capnonlimits\limits}
\excludecomment{verlong}
\includecomment{vershort}
\excludecomment{noncompile}

\newcommand{\timesu}{\mathbin{\underline{\times}}}

\renewcommand{\leq}{\leqslant}
\renewcommand{\geq}{\geqslant}

\newtheoremstyle{plainsl}
{8pt plus 2pt minus 4pt}
{8pt plus 2pt minus 4pt}
{\slshape}
{0pt}
{\bfseries}
{.}
{5pt plus 1pt minus 1pt}
{}
\theoremstyle{plainsl}
\ihead{Rank of Hankel matrices over finite fields}
\ohead{page \thepage}
\begin{document}

\title{On the rank of Hankel matrices over finite fields}
\author{Omesh Dhar Dwivedi, Darij Grinberg}
\date{September 11, 2021}
\maketitle

Given three nonnegative integers $p,q,r$ and a finite field $F$, how many
Hankel matrices $\left(  x_{i+j}\right)  _{0\leq i\leq p,\ 0\leq j\leq q}$
over $F$ have rank $\leq r$ ? This question is classical, and the answer
($q^{2r}$ when $r\leq\min\left\{  p,q\right\}  $) has been obtained
independently by various authors using different tools (\cite[Theorem 1 for
$m=n$]{Daykin}, \cite[(26)]{Elkies}, \cite[Theorem 5.1]{GaGhRa}). In this
note, we will study a refinement of this result: We will show that if we fix
the first $k$ of the entries $x_{0},x_{1},\ldots,x_{k-1}$ for some $k\leq
r\leq\min\left\{  p,q\right\}  $, then the number of ways to choose the
remaining $p+q-k+1$ entries $x_{k},x_{k+1},\ldots,x_{p+q}$ such that the
resulting Hankel matrix $\left(  x_{i+j}\right)  _{0\leq i\leq p,\ 0\leq j\leq
q}$ has rank $\leq r$ is $q^{2r-k}$. This is exactly the answer that one would
expect if the first $k$ entries had no effect on the rank, but of course the
situation is not this simple (and we had to combine some ideas from
\cite[(26)]{Elkies} and from \cite[Theorem 5.1 for $r=n$]{GaGhRa} to obtain
our proof). The refined result generalizes (and provides an alternative proof
of) \cite[Corollary 6.4]{Anzis-etc}.

\subsubsection*{Acknowledgments}

This note owes much to a short conversation between Peter Scholze and the
second author about a more conceptual reason for \cite[Corollary
6.4]{Anzis-etc} to be true. Peter, while having bigger fish to fry, quickly
ventured a guess, predicting that the known enumerative results for Hankel
matrices could be refined by fixing the first few entries. Despite his major
contribution, Peter declined to be a coauthor.

The second author is also grateful to the Mathematisches Forschungsinstitut
Oberwolfach, which hosted him as a Leibniz Fellow in 2020.

We also thank Jonah Blasiak for helpful and enlightening conversations.

\section{\label{sec.results}Results}

We let $\mathbb{N}$ denote the set $\left\{  0,1,2,\ldots\right\}  $.

Fix a field $F$. For any $n\in\mathbb{N}$, any $\left(  n+1\right)  $-tuple
$x=\left(  x_{0},x_{1},\ldots,x_{n}\right)  \in F^{n+1}$, and any two integers
$p,q\in\left\{  -1,0,1,\ldots\right\}  $ satisfying $p+q\leq n$, we define a
$\left(  p+1\right)  \times\left(  q+1\right)  $-matrix $H_{p,q}\left(
x\right)  $ by%
\[
H_{p,q}\left(  x\right)  :=\left(  x_{i+j}\right)  _{0\leq i\leq p,\ 0\leq
j\leq q}=\left(
\begin{array}
[c]{cccc}%
x_{0} & x_{1} & \cdots & x_{q}\\
x_{1} & x_{2} & \cdots & x_{q+1}\\
\vdots & \vdots & \ddots & \vdots\\
x_{p} & x_{p+1} & \cdots & x_{p+q}%
\end{array}
\right)  \in F^{\left(  p+1\right)  \times\left(  q+1\right)  }.
\]
Such a matrix $H_{p,q}\left(  x\right)  $ is called a \emph{Hankel matrix}.
The study of Hankel matrices has a long history in linear algebra (see, e.g.,
\cite{Iohvid82}) and relates to linearly recurrent sequences (\cite{Elkies},
\cite[\S 8.6]{LidNie}), coprime polynomials (\cite{GaGhRa}), determinants
(\cite[Section XII.II]{Muir}), orthogonal polynomials and continued fractions
(\cite[\S 2.7]{Krattenthaler}), total positivity (\cite{Khare}), and various
applications such as x-ray imaging (\cite[\S V.5]{Natterer}).\footnote{Some of
these references are studying \emph{Toeplitz matrices} instead of Hankel
matrices. However, this is equivalent, since a Toeplitz matrix is just a
Hankel matrix turned upside down (i.e., the result of reversing the order of
the rows in a Hankel matrix).} Numerous results have been obtained about their
ranks in particular (\cite[\S 11]{Iohvid82}). When the field $F$ is finite, a
strikingly simple formula can be given for the number of Hankel matrices of a
given rank (more precisely, of rank $\leq$ to a given number):

\begin{theorem}
\label{thm.hankel.mn<=r}Assume that $F$ is finite. Let $q=\left\vert
F\right\vert $. Let $r,m,n\in\mathbb{N}$ satisfy $r\leq m$ and $r\leq n$. The
number of $\left(  m+n+1\right)  $-tuples $x\in F^{m+n+1}$ satisfying
$\operatorname*{rank}\left(  H_{m,n}\left(  x\right)  \right)  \leq r$ is
$q^{2r}$.
\end{theorem}

\begin{example}
For a simple example, let $r=1$ and $m=2$ and $n=3$. Thus, for every
$x=\left(  x_{0},x_{1},x_{2},x_{3},x_{4},x_{5}\right)  \in F^{6}$, we have%
\[
H_{m,n}\left(  x\right)  =H_{2,3}\left(  x\right)  =\left(
\begin{array}
[c]{cccc}%
x_{0} & x_{1} & x_{2} & x_{3}\\
x_{1} & x_{2} & x_{3} & x_{4}\\
x_{2} & x_{3} & x_{4} & x_{5}%
\end{array}
\right)  .
\]
Theorem \ref{thm.hankel.mn<=r} yields that the number of $6$-tuples $x\in
F^{6}$ satisfying $\operatorname*{rank}\left(  H_{2,3}\left(  x\right)
\right)  \leq1$ is $q^{2\cdot1}=q^{2}$. These $6$-tuples can indeed be
described explicitly:

\begin{itemize}
\item Any $6$-tuple of the form $\left(  u,uv,uv^{2},uv^{3},uv^{4}%
,uv^{5}\right)  $ with $u\in F\setminus\left\{  0\right\}  $ and $v\in F$ is
such a $6$-tuple $x$. This gives a total of $\left\vert F\setminus\left\{
0\right\}  \right\vert \cdot\left\vert F\right\vert =\left(  q-1\right)  q$
many such $6$-tuples.

\item Any $6$-tuple of the form $\left(  0,0,0,0,0,w\right)  $ with $w\in F$
is such a $6$-tuple $x$. This gives a total of $\left\vert F\right\vert =q$
many such $6$-tuples.
\end{itemize}

For higher values of $r$, it is harder to describe all the $q^{2r}$ pertinent tuples.
\end{example}

To our knowledge, Theorem \ref{thm.hankel.mn<=r} has not appeared in this
exact form in the literature; however, it is easily seen to be equivalent to
the following variant, which has appeared in \cite[Theorem 1]{Daykin}:

\begin{corollary}
\label{cor.hankel.mn=r}Assume that $F$ is finite. Let $q=\left\vert
F\right\vert $. Let $r,m,n\in\mathbb{N}$ satisfy $m\leq n$. The number of
$\left(  m+n+1\right)  $-tuples $x\in F^{m+n+1}$ satisfying
$\operatorname*{rank}\left(  H_{m,n}\left(  x\right)  \right)  =r$ is
\[%
\begin{cases}
1, & \text{if }r=0;\\
q^{2r-2}\left(  q^{2}-1\right)  , & \text{if }0<r\leq m;\\
q^{2r-2}\left(  q^{n-m+1}-1\right)  , & \text{if }r=m+1;\\
0, & \text{if }r>m+1.
\end{cases}
\]

\end{corollary}

The particular case of Corollary \ref{cor.hankel.mn=r} for $m=n$ also appears
in \cite[Theorem 5.1]{GaGhRa}\footnote{Note that \cite[Theorem 5.1]{GaGhRa}
works with Toeplitz matrices instead of Hankel matrices, but this makes no
real difference, since a Toeplitz matrix is just a Hankel matrix turned upside
down (and this operation clearly does not change the rank of the matrix).} and
\cite[(26)]{Elkies}. The particular case when $r=m=n$ appears in
\cite[Corollary 3]{KalLob} as well.

Another setting in which Hankel matrices appear is the theory of symmetric
functions, specifically Schur functions (see, e.g., \cite[Chapter
7]{Stanley-EC2}). While we will not use this setting to prove our main
results, it has provided the main inspiration for this note, so we shall
briefly recall it now. The Jacobi--Trudi formula \cite[Theorem 7.16.1]%
{Stanley-EC2} expresses a Schur function $s_{\lambda}$ as the determinant of a
matrix, which is a Hankel matrix when the partition $\lambda$ is
rectangle-shaped. The recent result \cite[Corollary 6.4]{Anzis-etc} by Anzis,
Chen, Gao, Kim, Li and Patrias can thus be framed as a formula for the
probability of a certain $\left(  n+1\right)  \times\left(  n+1\right)  $
Hankel matrix over a finite field to have determinant $0$ (that is, rank $\leq
n$). This would be a particular case of Theorem \ref{thm.hankel.mn<=r} if not
for the fact that the entries of the relevant Hankel matrix are not chosen
uniformly at random; instead, the first few of them are fixed, while the rest
are chosen uniformly at random\footnote{See Section \ref{sec.jt} for concrete
examples of such matrices.}. This suggests a generalization of Theorem
\ref{thm.hankel.mn<=r} in which the first few entries\footnote{Specifically,
\textquotedblleft first few\textquotedblright\ means \textquotedblleft at most
$m$\textquotedblright.} of the $\left(  m+n+1\right)  $-tuples $x\in
F^{m+n+1}$ are fixed. The existence of such a generalization was suggested to
us by Peter Scholze.

This generalization indeed exists, and will be the main result of this note.
In stating it, we will use the following notation:

\begin{definition}
Let $n\in\mathbb{N}$. Let $x=\left(  x_{0},x_{1},\ldots,x_{n}\right)  $ be any
$\left(  n+1\right)  $-tuple of any kinds of objects. Let $i\in\left\{
0,1,\ldots,n+1\right\}  $. Then, $x_{\left[  0,i\right)  }$ denotes the
$i$-tuple $\left(  x_{0},x_{1},\ldots,x_{i-1}\right)  $.
\end{definition}

For instance, $\left(  a,b,c,d,e\right)  _{\left[  0,3\right)  }=\left(
a,b,c\right)  $.

We can now state our generalization of Theorem \ref{thm.hankel.mn<=r}:

\begin{theorem}
\label{thm.hankel.mn<=rk}Assume that $F$ is finite. Let $q=\left\vert
F\right\vert $. Let $k,r,m,n\in\mathbb{N}$ satisfy $k\leq r\leq m$ and $r\leq
n$. Fix any $k$-tuple $a=\left(  a_{0},a_{1},\ldots,a_{k-1}\right)  \in F^{k}%
$. The number of $\left(  m+n+1\right)  $-tuples $x\in F^{m+n+1}$ satisfying
$x_{\left[  0,k\right)  }=a$ and $\operatorname*{rank}\left(  H_{m,n}\left(
x\right)  \right)  \leq r$ is $q^{2r-k}$.
\end{theorem}

\begin{example}
For an example, let $k=2$, $r=2$, $m=3$ and $n=3$. Let $a=\left(  a_{0}%
,a_{1}\right)  \in F^{2}$. Then, Theorem \ref{thm.hankel.mn<=rk} yields that
the number of $7$-tuples $x\in F^{7}$ satisfying $x_{\left[  0,2\right)  }=a$
and $\operatorname*{rank}\left(  H_{3,3}\left(  x\right)  \right)  \leq2$ is
$q^{2\cdot2-2}=q^{2}$. Note that a $7$-tuple $x\in F^{7}$ satisfying
$x_{\left[  0,2\right)  }=a$ is nothing but a $7$-tuple $x\in F^{7}$ that
begins with the entries $a_{0}$ and $a_{1}$; thus, we could just as well be
counting the $5$-tuples $\left(  x_{2},x_{3},x_{4},x_{5},x_{6}\right)  \in
F^{5}$ satisfying $\operatorname*{rank}\left(  H_{3,3}\left(  a_{0}%
,a_{1},x_{2},x_{3},x_{4},x_{5},x_{6}\right)  \right)  \leq2$.
\end{example}

Clearly, Theorem \ref{thm.hankel.mn<=r} is the particular case of Theorem
\ref{thm.hankel.mn<=rk} for $k=0$, since the $0$-tuple $a=\left(  {}\right)
\in F^{0}$ automatically satisfies $x_{\left[  0,0\right)  }=a$ for every
$x\in F^{m+n+1}$.

By specializing Theorem \ref{thm.hankel.mn<=r} to the case $r=m=n$ (and
recalling that a square matrix has determinant $0$ if and only if it has
less-than-full rank), we can easily obtain the following:

\begin{corollary}
\label{cor.hankel.detk}Assume that $F$ is finite. Let $q=\left\vert
F\right\vert $. Let $k,n\in\mathbb{N}$ satisfy $k\leq n$. Fix any $k$-tuple
$a=\left(  a_{0},a_{1},\ldots,a_{k-1}\right)  \in F^{k}$. The number of
$\left(  2n+1\right)  $-tuples $x\in F^{2n+1}$ satisfying $x_{\left[
0,k\right)  }=a$ and $\det\left(  H_{n,n}\left(  x\right)  \right)  =0$ is
$q^{2n-k}$.
\end{corollary}

We shall prove Theorem \ref{thm.hankel.mn<=rk} in Section
\ref{sec.proofs-main}; we will then derive Theorem \ref{thm.hankel.mn<=r},
Corollary \ref{cor.hankel.mn=r} and Corollary \ref{cor.hankel.detk} from it.
Finally, in Section \ref{sec.jt}, we will explain how Corollary
\ref{cor.hankel.detk} generalizes \cite[Corollary 6.4]{Anzis-etc}.

\begin{remark}
Theorem \ref{thm.hankel.mn<=rk} also holds if we replace the assumptions
\textquotedblleft$k\leq r\leq m$ and $r\leq n$\textquotedblright\ by
\textquotedblleft$k\leq r\leq m\leq n+1$\textquotedblright. In fact, the only
case covered by the latter assumptions but not by the former is the case when
$k\leq r=m=n+1$; however, Theorem \ref{thm.hankel.mn<=rk} is easy to prove
directly in this case. (To wit, if $k\leq r=m=n+1$, then \textbf{every}
$\left(  m+n+1\right)  $-tuple $x\in F^{m+n+1}$ satisfies
$\operatorname*{rank}\left(  H_{m,n}\left(  x\right)  \right)  \leq r$, since
the matrix $H_{m,n}\left(  x\right)  $ has $n+1$ columns and therefore has
rank $\leq n+1=r$. Hence, the number of $\left(  m+n+1\right)  $-tuples $x\in
F^{m+n+1}$ satisfying $x_{\left[  0,k\right)  }=a$ and $\operatorname*{rank}%
\left(  H_{m,n}\left(  x\right)  \right)  \leq r$ equals the number of
\textbf{all} $\left(  m+n+1\right)  $-tuples $x\in F^{m+n+1}$ satisfying
$x_{\left[  0,k\right)  }=a$ in this case. But this number is easily seen to
be $q^{m+n+1-k}=q^{2r-k}$ (since $\underbrace{m}_{=r}+\underbrace{n+1}%
_{=r}=r+r=2r$). Thus, Theorem \ref{thm.hankel.mn<=rk} is proved in the case
when $k\leq r=m=n+1$.)

As a consequence, Theorem \ref{thm.hankel.mn<=r} also holds if we replace the
assumptions \textquotedblleft$r\leq m$ and $r\leq n$\textquotedblright\ by
\textquotedblleft$r\leq m\leq n+1$\textquotedblright. Hence, Corollary
\ref{cor.hankel.mn=r} still holds if we replace the assumption
\textquotedblleft$m\leq n$\textquotedblright\ by \textquotedblleft$m\leq
n+1$\textquotedblright. However, we gain nothing significantly new in this
way, since the newly covered cases can also be easily obtained from the old ones.
\end{remark}

\section{\label{sec.hank-ranks}Rank lemmas}

Before we come to the proof of Theorem \ref{thm.hankel.mn<=rk}, we are going
to build a toolbox of general lemmas about ranks of the Hankel matrices
$H_{p,q}\left(  x\right)  $. We note that none of these lemmas requires $F$ to
be finite; they can equally well be applied to fields like $\mathbb{R}$ and
$\mathbb{C}$.

\begin{lemma}
\label{lem.hank-ranks.1}Let $n\in\mathbb{N}$. Let $p,q\in\mathbb{N}$ be such
that $p+q\leq n+1$. If $x\in F^{n+1}$ satisfies $\operatorname*{rank}\left(
H_{p,q-1}\left(  x\right)  \right)  \leq p$, then%
\[
\operatorname*{rank}\left(  H_{p,q-1}\left(  x\right)  \right)  \leq
\operatorname*{rank}\left(  H_{p-1,q}\left(  x\right)  \right)  .
\]

\end{lemma}

\begin{proof}
[Proof of Lemma \ref{lem.hank-ranks.1}.]We proceed by induction on $p$
(without fixing $x$):

\textit{Induction base:} Proving Lemma \ref{lem.hank-ranks.1} in the case when
$p=0$ is easy: In this case, the assumption $\operatorname*{rank}\left(
H_{p,q-1}\left(  x\right)  \right)  \leq p$ rewrites as $\operatorname*{rank}%
\left(  H_{p,q-1}\left(  x\right)  \right)  \leq0$, which immediately yields
the claim.

\textit{Induction step:} Let $p$ be a positive integer. Assume (as the
induction hypothesis) that Lemma \ref{lem.hank-ranks.1} holds for $p-1$
instead of $p$. Our goal is now to prove Lemma \ref{lem.hank-ranks.1} for $p$.

Let $q\in\mathbb{N}$ be such that $p+q\leq n+1$. Let $x\in F^{n+1}$ satisfy
$\operatorname*{rank}\left(  H_{p,q-1}\left(  x\right)  \right)  \leq p$. We
must thus prove that%
\[
\operatorname*{rank}\left(  H_{p,q-1}\left(  x\right)  \right)  \leq
\operatorname*{rank}\left(  H_{p-1,q}\left(  x\right)  \right)  .
\]

Write the $\left(  n+1\right)  $-tuple $x\in F^{n+1}$ as $x=\left(
x_{0},x_{1},\ldots,x_{n}\right)  $. Then,%
\begin{align*}
H_{p,q-1}\left(  x\right)   &  =\left(
\begin{array}
[c]{cccc}%
x_{0} & x_{1} & \cdots & x_{q-1}\\
x_{1} & x_{2} & \cdots & x_{q}\\
\vdots & \vdots & \ddots & \vdots\\
x_{p} & x_{p+1} & \cdots & x_{p+q-1}%
\end{array}
\right)  \ \ \ \ \ \ \ \ \ \ \text{and}\\
H_{p-1,q}\left(  x\right)   &  =\left(
\begin{array}
[c]{cccc}%
x_{0} & x_{1} & \cdots & x_{q}\\
x_{1} & x_{2} & \cdots & x_{q+1}\\
\vdots & \vdots & \ddots & \vdots\\
x_{p-1} & x_{p} & \cdots & x_{p+q-1}%
\end{array}
\right)  \ \ \ \ \ \ \ \ \ \ \text{and}\\
H_{p-1,q-1}\left(  x\right)   &  =\left(
\begin{array}
[c]{cccc}%
x_{0} & x_{1} & \cdots & x_{q-1}\\
x_{1} & x_{2} & \cdots & x_{q}\\
\vdots & \vdots & \ddots & \vdots\\
x_{p-1} & x_{p} & \cdots & x_{p+q-2}%
\end{array}
\right)  .
\end{align*}
Hence, the matrix $H_{p,q-1}\left(  x\right)  $ is $H_{p-1,q-1}\left(
x\right)  $ with one extra row inserted at the bottom, whereas the matrix
$H_{p-1,q}\left(  x\right)  $ is $H_{p-1,q-1}\left(  x\right)  $ with one
extra column inserted at the right end.

For any matrix $A$ that has at least one row, we let $\overline{A}$ denote the
matrix $A$ with its first row removed. The following properties of
$\overline{A}$ are well-known:

\begin{itemize}
\item If the first row of $A$ is a linear combination of the remaining rows,
then%
\begin{equation}
\operatorname*{rank}A=\operatorname*{rank}\overline{A}.
\label{pf.lem.hank-ranks.1.rankA0}%
\end{equation}

\item If the first row of $A$ is \textbf{not} a linear combination of the
remaining rows, then%
\begin{equation}
\operatorname*{rank}A=\operatorname*{rank}\overline{A}+1.
\label{pf.lem.hank-ranks.1.rankA1}%
\end{equation}

\end{itemize}

It is furthermore well-known that if $A$ is any matrix, and if $B$ is any
submatrix of $A$, then $\operatorname*{rank}B\leq\operatorname*{rank}A$.
However, the matrix $\overline{H_{p,q-1}\left(  x\right)  }$ is a submatrix of
$H_{p-1,q}\left(  x\right)  $ (indeed, it can be obtained from $H_{p-1,q}%
\left(  x\right)  $ by removing the first column). Hence,
\[
\operatorname*{rank}\left(  \overline{H_{p,q-1}\left(  x\right)  }\right)
\leq\operatorname*{rank}\left(  H_{p-1,q}\left(  x\right)  \right)  .
\]

Let $\overline{x}$ denote the $n$-tuple $\left(  x_{1},x_{2},\ldots
,x_{n}\right)  \in F^{n}$. It is easy to see that
\begin{equation}
\overline{H_{u,v}\left(  x\right)  }=H_{u-1,v}\left(  \overline{x}\right)
\label{pf.lem.hank-ranks.1.Huvbar}%
\end{equation}
for all $u\in\mathbb{N}$ and $v\in\left\{  -1,0,1,\ldots\right\}  $ satisfying
$u+v\leq n$. Thus, in particular,%
\begin{equation}
\overline{H_{p,q-1}\left(  x\right)  }=H_{p-1,q-1}\left(  \overline{x}\right)
\label{pf.lem.hank-ranks.1.Huvbar1}%
\end{equation}
and%
\begin{equation}
\overline{H_{p-1,q}\left(  x\right)  }=H_{p-2,q}\left(  \overline{x}\right)  .
\label{pf.lem.hank-ranks.1.Huvbar2}%
\end{equation}

If the first row of the matrix $H_{p,q-1}\left(  x\right)  $ is a linear
combination of the remaining rows, then (\ref{pf.lem.hank-ranks.1.rankA0})
yields
\[
\operatorname*{rank}\left(  H_{p,q-1}\left(  x\right)  \right)
=\operatorname*{rank}\left(  \overline{H_{p,q-1}\left(  x\right)  }\right)
\leq\operatorname*{rank}\left(  H_{p-1,q}\left(  x\right)  \right)  ,
\]
which is precisely what we wanted to show. Hence, for the rest of this proof,
we WLOG assume that the first row of the matrix $H_{p,q-1}\left(  x\right)  $
is \textbf{not} a linear combination of the remaining rows. Thus,
(\ref{pf.lem.hank-ranks.1.rankA1}) yields%
\[
\operatorname*{rank}\left(  H_{p,q-1}\left(  x\right)  \right)
=\operatorname*{rank}\left(  \overline{H_{p,q-1}\left(  x\right)  }\right)
+1.
\]
In view of (\ref{pf.lem.hank-ranks.1.Huvbar1}), this rewrites as%
\begin{equation}
\operatorname*{rank}\left(  H_{p,q-1}\left(  x\right)  \right)
=\operatorname*{rank}\left(  H_{p-1,q-1}\left(  \overline{x}\right)  \right)
+1. \label{pf.lem.hank-ranks.1.plus1-a}%
\end{equation}
Hence,
\[
\operatorname*{rank}\left(  H_{p-1,q-1}\left(  \overline{x}\right)  \right)
=\underbrace{\operatorname*{rank}\left(  H_{p,q-1}\left(  x\right)  \right)
}_{\leq p}-1\leq p-1.
\]

Recall that the first row of the matrix $H_{p,q-1}\left(  x\right)  $ is not a
linear combination of the remaining rows. This entails that the first row of
the matrix $H_{p-1,q-1}\left(  x\right)  $ is not a linear combination of the
remaining rows (since the matrix $H_{p-1,q-1}\left(  x\right)  $ is the same
as $H_{p,q-1}\left(  x\right)  $ without the last row). Therefore, the first
row of the matrix $H_{p-1,q}\left(  x\right)  $ is not a linear combination of
the remaining rows (since the matrix $H_{p-1,q}\left(  x\right)  $ is just
$H_{p-1,q-1}\left(  x\right)  $ with an extra column). Thus,
(\ref{pf.lem.hank-ranks.1.rankA1}) yields%
\[
\operatorname*{rank}\left(  H_{p-1,q}\left(  x\right)  \right)
=\operatorname*{rank}\left(  \overline{H_{p-1,q}\left(  x\right)  }\right)
+1.
\]
In view of (\ref{pf.lem.hank-ranks.1.Huvbar2}), this rewrites as%
\begin{equation}
\operatorname*{rank}\left(  H_{p-1,q}\left(  x\right)  \right)
=\operatorname*{rank}\left(  H_{p-2,q}\left(  \overline{x}\right)  \right)
+1. \label{pf.lem.hank-ranks.1.plus1-b}%
\end{equation}

However, our induction hypothesis shows that we can apply Lemma
\ref{lem.hank-ranks.1} to $n-1$, $p-1$ and $\overline{x}$ instead of $n$, $p$
and $x$ (since $\operatorname*{rank}\left(  H_{p-1,q-1}\left(  \overline
{x}\right)  \right)  \leq p-1$). We thus obtain
\[
\operatorname*{rank}\left(  H_{p-1,q-1}\left(  \overline{x}\right)  \right)
\leq\operatorname*{rank}\left(  H_{p-2,q}\left(  \overline{x}\right)  \right)
.
\]
Adding $1$ to both sides of this inequality, we find%
\[
\operatorname*{rank}\left(  H_{p-1,q-1}\left(  \overline{x}\right)  \right)
+1\leq\operatorname*{rank}\left(  H_{p-2,q}\left(  \overline{x}\right)
\right)  +1.
\]
In view of (\ref{pf.lem.hank-ranks.1.plus1-a}) and
(\ref{pf.lem.hank-ranks.1.plus1-b}), this rewrites as $\operatorname*{rank}%
\left(  H_{p,q-1}\left(  x\right)  \right)  \leq\operatorname*{rank}\left(
H_{p-1,q}\left(  x\right)  \right)  $. This completes the induction step.
Thus, Lemma \ref{lem.hank-ranks.1} is proved.
\end{proof}

\begin{lemma}
\label{lem.hank-ranks.2}Let $n\in\mathbb{N}$. Let $p,q\in\mathbb{N}$ be such
that $p+q\leq n+1$. If $x\in F^{n+1}$ satisfies $\operatorname*{rank}\left(
H_{p-1,q}\left(  x\right)  \right)  \leq q$, then%
\[
\operatorname*{rank}\left(  H_{p-1,q}\left(  x\right)  \right)  \leq
\operatorname*{rank}\left(  H_{p,q-1}\left(  x\right)  \right)  .
\]

\end{lemma}

\begin{proof}
[Proof of Lemma \ref{lem.hank-ranks.2}.]This is just a restatement of Lemma
\ref{lem.hank-ranks.1} (applied to $p$ and $q$ instead of $q$ and $p$), since
the matrices $H_{p-1,q}\left(  x\right)  $ and $H_{p,q-1}\left(  x\right)  $
are the transposes of the matrices $H_{q,p-1}\left(  x\right)  $ and
$H_{q-1,p}\left(  x\right)  $. (Alternatively, you can prove it by the same
argument as we used to prove Lemma \ref{lem.hank-ranks.1}, except that rows
and columns switch roles.)
\end{proof}

\begin{lemma}
\label{lem.hank-ranks.3}Let $n\in\mathbb{N}$. Let $p,q\in\mathbb{N}$ be such
that $p+q\leq n+1$. If $x\in F^{n+1}$ satisfies $\operatorname*{rank}\left(
H_{p,q-1}\left(  x\right)  \right)  \leq p$ and $\operatorname*{rank}\left(
H_{p-1,q}\left(  x\right)  \right)  \leq q$, then%
\[
\operatorname*{rank}\left(  H_{p,q-1}\left(  x\right)  \right)
=\operatorname*{rank}\left(  H_{p-1,q}\left(  x\right)  \right)  .
\]

\end{lemma}

\begin{proof}
[Proof of Lemma \ref{lem.hank-ranks.3}.]This follows by combining Lemma
\ref{lem.hank-ranks.1} with Lemma \ref{lem.hank-ranks.2}.
\end{proof}

Our next lemma is a simple corollary of Lemma \ref{lem.hank-ranks.3}:

\begin{lemma}
\label{lem.hank-ranks.5}Let $n\in\mathbb{N}$. Let $p,q\in\mathbb{N}$ be such
that $p+q\leq n+1$. Let $r\in\mathbb{N}$ satisfy $r+1\leq p$ and $r+1\leq q$.
Let $x\in F^{n+1}$. Then, we have the logical equivalence%
\[
\left(  \operatorname*{rank}\left(  H_{p,q-1}\left(  x\right)  \right)  \leq
r\right)  \ \Longleftrightarrow\ \left(  \operatorname*{rank}\left(
H_{p-1,q}\left(  x\right)  \right)  \leq r\right)  .
\]

\end{lemma}

\begin{proof}
[Proof of Lemma \ref{lem.hank-ranks.5}.]We must prove the two implications%
\begin{equation}
\left(  \operatorname*{rank}\left(  H_{p,q-1}\left(  x\right)  \right)  \leq
r\right)  \ \Longrightarrow\ \left(  \operatorname*{rank}\left(
H_{p-1,q}\left(  x\right)  \right)  \leq r\right)
\label{pf.lem.hank-ranks.5.goal1}%
\end{equation}
and%
\begin{equation}
\left(  \operatorname*{rank}\left(  H_{p-1,q}\left(  x\right)  \right)  \leq
r\right)  \ \Longrightarrow\ \left(  \operatorname*{rank}\left(
H_{p,q-1}\left(  x\right)  \right)  \leq r\right)  .
\label{pf.lem.hank-ranks.5.goal2}%
\end{equation}
We shall only prove (\ref{pf.lem.hank-ranks.5.goal1}), since
(\ref{pf.lem.hank-ranks.5.goal2}) is entirely analogous.

So let us prove (\ref{pf.lem.hank-ranks.5.goal1}). We assume that
$\operatorname*{rank}\left(  H_{p,q-1}\left(  x\right)  \right)  \leq r$; we
then must show that $\operatorname*{rank}\left(  H_{p-1,q}\left(  x\right)
\right)  \leq r$.

The matrix $H_{p-1,q-1}\left(  x\right)  $ is a submatrix of $H_{p,q-1}\left(
x\right)  $, and thus its rank cannot surpass the rank of $H_{p,q-1}\left(
x\right)  $. In other words, we have $\operatorname*{rank}\left(
H_{p-1,q-1}\left(  x\right)  \right)  \leq\operatorname*{rank}\left(
H_{p,q-1}\left(  x\right)  \right)  $.

However, the matrix $H_{p-1,q}\left(  x\right)  $ can be viewed as being the
matrix $H_{p-1,q-1}\left(  x\right)  $ with one extra column attached to it
(at its right end). Thus,%
\[
\operatorname*{rank}\left(  H_{p-1,q}\left(  x\right)  \right)  \leq
\operatorname*{rank}\left(  H_{p-1,q-1}\left(  x\right)  \right)  +1
\]
(since attaching one column cannot increase the rank of a matrix by more than
$1$). Hence,%
\[
\operatorname*{rank}\left(  H_{p-1,q}\left(  x\right)  \right)  \leq
\underbrace{\operatorname*{rank}\left(  H_{p-1,q-1}\left(  x\right)  \right)
}_{\leq\operatorname*{rank}\left(  H_{p,q-1}\left(  x\right)  \right)  \leq
r}+1\leq r+1\leq q.
\]
Moreover, $\operatorname*{rank}\left(  H_{p,q-1}\left(  x\right)  \right)
\leq r\leq r+1\leq p$. Hence, we can apply Lemma \ref{lem.hank-ranks.3}, and
conclude that $\operatorname*{rank}\left(  H_{p,q-1}\left(  x\right)  \right)
=\operatorname*{rank}\left(  H_{p-1,q}\left(  x\right)  \right)  $. Thus, of
course, $\operatorname*{rank}\left(  H_{p-1,q}\left(  x\right)  \right)  \leq
r$ follows immediately from our assumption $\operatorname*{rank}\left(
H_{p,q-1}\left(  x\right)  \right)  \leq r$. Hence,
(\ref{pf.lem.hank-ranks.5.goal1}) is proved.

As we said, the proof of (\ref{pf.lem.hank-ranks.5.goal2}) is analogous. Thus,
the proof of Lemma \ref{lem.hank-ranks.5} is complete.
\end{proof}

The following lemma is a (much simpler) counterpart to Lemma
\ref{lem.hank-ranks.1} that replaces the assumption $\operatorname*{rank}%
\left(  H_{p,q-1}\left(  x\right)  \right)  \leq p$ by the reverse inequality:

\begin{lemma}
\label{lem.hank-ranks.4}Let $n\in\mathbb{N}$. Let $p,q\in\mathbb{N}$ be such
that $p+q\leq n+1$. If $x\in F^{n+1}$ satisfies $\operatorname*{rank}\left(
H_{p,q-1}\left(  x\right)  \right)  >p$, then%
\[
\operatorname*{rank}\left(  H_{p-1,q}\left(  x\right)  \right)  =p.
\]

\end{lemma}

\begin{proof}
[Proof of Lemma \ref{lem.hank-ranks.4}.]Let $x\in F^{n+1}$ satisfy
$\operatorname*{rank}\left(  H_{p,q-1}\left(  x\right)  \right)  >p$. The
assumption $\operatorname*{rank}\left(  H_{p,q-1}\left(  x\right)  \right)
>p$ shows that the $p+1$ rows of the matrix $H_{p,q-1}\left(  x\right)  $ are
linearly independent. Hence, in particular, the $p$ rows of the matrix
$H_{p-1,q-1}\left(  x\right)  $ are linearly independent (since these $p$ rows
are simply the first $p$ rows of the matrix $H_{p,q-1}\left(  x\right)  $).
Therefore, the $p$ rows of the matrix $H_{p-1,q}\left(  x\right)  $ are
linearly independent as well (since the matrix $H_{p-1,q}\left(  x\right)  $
is just $H_{p-1,q-1}\left(  x\right)  $ with an extra column, and therefore
the rows of the former contain the rows of the latter as subsequences). In
other words, $\operatorname*{rank}\left(  H_{p-1,q}\left(  x\right)  \right)
=p$. This proves Lemma \ref{lem.hank-ranks.4}.
\end{proof}

Our above lemmas have related ranks of the \textquotedblleft
adjacent\textquotedblright\ Hankel matrices $\operatorname*{rank}\left(
H_{p,q-1}\left(  x\right)  \right)  $ and $\operatorname*{rank}\left(
H_{p-1,q}\left(  x\right)  \right)  $. By induction, we shall now extend these
to further-apart Hankel matrices:

\begin{lemma}
\label{lem.hank-ranks.6}Let $u\in\mathbb{N}$. Let $m,n,r\in\mathbb{N}$ be such
that $m+n\leq u$ and $r\leq m$ and $r\leq n$. Let $s=m+n-r$. Let $x\in
F^{u+1}$ be arbitrary. Then, we have the logical equivalence%
\[
\left(  \operatorname*{rank}\left(  H_{m,n}\left(  x\right)  \right)  \leq
r\right)  \ \Longleftrightarrow\ \left(  \operatorname*{rank}\left(
H_{r,s}\left(  x\right)  \right)  \leq r\right)  .
\]

\end{lemma}

Before we prove this lemma, let us comment on its significance (even though we
will use it rather directly): If one wants to determine the rank of a $\left(
m+1\right)  \times\left(  n+1\right)  $-matrix $A$, it suffices to probe for
each $r\in\left\{  0,1,\ldots,\min\left\{  m,n\right\}  \right\}  $ whether
$\operatorname*{rank}A\leq r$ is true (since $0\leq\operatorname*{rank}%
A\leq\min\left\{  m,n\right\}  +1$). Thus, Lemma \ref{lem.hank-ranks.6} allows
us to determine the ranks of the various matrices $H_{m,n}\left(  x\right)  $
for a given $x\in F^{u+1}$ if we know which pairs $\left(  r,s\right)  $
satisfy $\operatorname*{rank}\left(  H_{r,s}\left(  x\right)  \right)  \leq r$.

\begin{proof}
[Proof of Lemma \ref{lem.hank-ranks.6}.]From $s=m+n-r=\left(  m-r\right)  +n$,
we obtain $s-\left(  m-r\right)  =n$. Furthermore, $s=m+n-\underbrace{r}_{\leq
m}\geq m+n-m=n$ and similarly $s\geq m$. Hence, $m\leq s$.

We now claim that the equivalence%
\begin{equation}
\left(  \operatorname*{rank}\left(  H_{r+i,s-i}\left(  x\right)  \right)  \leq
r\right)  \ \Longleftrightarrow\ \left(  \operatorname*{rank}\left(
H_{r,s}\left(  x\right)  \right)  \leq r\right)
\label{pf.lem.hank-ranks.6.eq}%
\end{equation}
holds for each $i\in\left\{  0,1,\ldots,s-r\right\}  $.

[\textit{Proof of (\ref{pf.lem.hank-ranks.6.eq}):} We proceed by induction on
$i$:

\textit{Induction base:} Clearly, (\ref{pf.lem.hank-ranks.6.eq}) holds for
$i=0$, since we have $H_{r+i,s-i}\left(  x\right)  =H_{r+0,s-0}\left(
x\right)  =H_{r,s}\left(  x\right)  $ in this case.

\textit{Induction step:} Let $j\in\left\{  1,2,\ldots,s-r\right\}  $. Assume
(as the induction hypothesis) that (\ref{pf.lem.hank-ranks.6.eq}) holds for
$i=j-1$. We must prove that (\ref{pf.lem.hank-ranks.6.eq}) holds for $i=j$. In
other words, we must prove the equivalence%
\begin{equation}
\left(  \operatorname*{rank}\left(  H_{r+j,s-j}\left(  x\right)  \right)  \leq
r\right)  \ \Longleftrightarrow\ \left(  \operatorname*{rank}\left(
H_{r,s}\left(  x\right)  \right)  \leq r\right)  .
\label{pf.lem.hank-ranks.6.eq.pf.goal}%
\end{equation}

However, our induction hypothesis tells us that the equivalence%
\begin{equation}
\left(  \operatorname*{rank}\left(  H_{r+\left(  j-1\right)  ,s-\left(
j-1\right)  }\left(  x\right)  \right)  \leq r\right)  \ \Longleftrightarrow
\ \left(  \operatorname*{rank}\left(  H_{r,s}\left(  x\right)  \right)  \leq
r\right)  \label{pf.lem.hank-ranks.6.eq.pf.1}%
\end{equation}
holds.

We have
\[
\left(  r+j\right)  +\left(  s-j+1\right)  =r+\underbrace{s}_{=m+n-r}%
+1=r+\left(  m+n-r\right)  +1=\underbrace{m+n}_{\leq u}+1\leq u+1.
\]
Furthermore, we have $j\in\left\{  1,2,\ldots,s-r\right\}  $, so that $1\leq
j\leq s-r$. From $j\leq s-r$, we obtain $r\leq s-j$, so that $r+1\leq s-j+1$.
This entails $s-j+1\in\mathbb{N}$ (since $r+1\in\mathbb{N}$). Also, $r+1\leq
r+j$ (since $j\geq1$). Hence, Lemma \ref{lem.hank-ranks.5} (applied to $p=r+j$
and $n=s-j+1$) yields that we have the logical equivalence%
\[
\left(  \operatorname*{rank}\left(  H_{r+j,s-j+1-1}\left(  x\right)  \right)
\leq r\right)  \ \Longleftrightarrow\ \left(  \operatorname*{rank}\left(
H_{r+j-1,s-j+1}\left(  x\right)  \right)  \leq r\right)  .
\]
In other words, we have the equivalence%
\[
\left(  \operatorname*{rank}\left(  H_{r+j,s-j}\left(  x\right)  \right)  \leq
r\right)  \ \Longleftrightarrow\ \left(  \operatorname*{rank}\left(
H_{r+\left(  j-1\right)  ,s-\left(  j-1\right)  }\left(  x\right)  \right)
\leq r\right)
\]
(since $s-j+1-1=s-j$ and $r+j-1=r+\left(  j-1\right)  $ and $s-j+1=s-\left(
j-1\right)  $). Combining this equivalence with
(\ref{pf.lem.hank-ranks.6.eq.pf.1}), we obtain precisely the equivalence
(\ref{pf.lem.hank-ranks.6.eq.pf.goal}) that we were meaning to prove.

Thus, we have shown that (\ref{pf.lem.hank-ranks.6.eq}) holds for $i=j$. This
completes the induction step, so that (\ref{pf.lem.hank-ranks.6.eq}) is proven.]

Now, $m\in\left\{  r,r+1,\ldots,s\right\}  $ (since $r\leq m\leq s$), so that
$m-r\in\left\{  0,1,\ldots,s-r\right\}  $. Hence, we can apply
(\ref{pf.lem.hank-ranks.6.eq}) to $i=m-r$. As a result, we obtain that the
equivalence%
\[
\left(  \operatorname*{rank}\left(  H_{r+m-r,s-\left(  m-r\right)  }\left(
x\right)  \right)  \leq r\right)  \ \Longleftrightarrow\ \left(
\operatorname*{rank}\left(  H_{r,s}\left(  x\right)  \right)  \leq r\right)
\]
holds. In other words, the equivalence%
\[
\left(  \operatorname*{rank}\left(  H_{m,n}\left(  x\right)  \right)  \leq
r\right)  \ \Longleftrightarrow\ \left(  \operatorname*{rank}\left(
H_{r,s}\left(  x\right)  \right)  \leq r\right)
\]
holds (since $r+m-r=m$ and $s-\left(  m-r\right)  =n$). This proves Lemma
\ref{lem.hank-ranks.6}.
\end{proof}

\section{\label{sec.main-lem}Auxiliary enumerative results}

\subsection{Assumptions and notations}

From now on, we assume that the field $F$ is finite. We set $q=\left\vert
F\right\vert $.

We shall use the so-called
\href{https://en.wikipedia.org/wiki/Iverson_bracket}{\textit{Iverson bracket
notation}}:

\begin{definition}
\label{def.iverson}If $\mathcal{A}$ is any logical statement, then we define
an integer $\left[  \mathcal{A}\right]  \in\left\{  0,1\right\}  $ by%
\[
\left[  \mathcal{A}\right]  =%
\begin{cases}
1, & \text{if }\mathcal{A}\text{ is true};\\
0, & \text{if }\mathcal{A}\text{ is false}.
\end{cases}
\]
For example, $\left[  2+2=4\right]  =1$ but $\left[  2+2=5\right]  =0$.

If $\mathcal{A}$ is any logical statement, then $\left[  \mathcal{A}\right]  $
is known as the \emph{truth value} of $\mathcal{A}$.
\end{definition}

The following fact (\textquotedblleft counting by roll-call\textquotedblright)
makes truth values useful to us:

\begin{proposition}
\label{prop.iverson.rollcall}Let $S$ be a finite set. Let $\mathcal{A}\left(
s\right)  $ be a logical statement for each $s\in S$. Then, $\sum_{s\in
S}\left[  \mathcal{A}\left(  s\right)  \right]  $ equals the number of
elements $s\in S$ satisfying $\mathcal{A}\left(  s\right)  $.
\end{proposition}

\subsection{Sums over $v$ for fixed $x$}

The following proposition is a restatement of \cite[Proposition 2]{Elkies}
(but we shall prove it nevertheless to keep this note self-contained):

\begin{proposition}
\label{prop.elkies-prop2}Let $m,n\in\mathbb{N}$ satisfy $m\leq n+1$. Let $x\in
F^{m+n+1}$ be a $\left(  m+n+1\right)  $-tuple. Then,%
\begin{align*}
&  \left(  q-1\right)  \cdot\left[  \operatorname*{rank}\left(  H_{m,n}\left(
x\right)  \right)  \leq m\right] \\
&  =\sum_{\substack{v\in F^{1\times\left(  m+1\right)  };\\v\neq0}}\left[
v\ H_{m,n}\left(  x\right)  =0\right]  -q\sum_{\substack{v\in F^{1\times
m};\\v\neq0}}\left[  v\ H_{m-1,n+1}\left(  x\right)  =0\right]  .
\end{align*}

\end{proposition}

Before we prove this proposition, a few words about its significance are worth
saying. Assume that, as a first step towards proving Theorem
\ref{thm.hankel.mn<=rk}, we want to count the $\left(  m+n+1\right)  $-tuples
$x\in F^{m+n+1}$ satisfying $\operatorname*{rank}\left(  H_{m,n}\left(
x\right)  \right)  \leq m$. (This is just an interim goal; we will later
generalize this inequality to $\operatorname*{rank}\left(  H_{m,n}\left(
x\right)  \right)  \leq r$ and impose the additional condition $x_{\left[
0,k\right)  }=a$.) In view of Proposition \ref{prop.iverson.rollcall}, this
boils down to computing $\sum_{x\in F^{m+n+1}}\left[  \operatorname*{rank}%
\left(  H_{m,n}\left(  x\right)  \right)  \leq m\right]  $. Using Proposition
\ref{prop.elkies-prop2}, we can rewrite the addends $\left[
\operatorname*{rank}\left(  H_{m,n}\left(  x\right)  \right)  \leq m\right]  $
in this sum in terms of other truth values, which are more \textquotedblleft
local\textquotedblright\ (one can think of \textquotedblleft%
$\operatorname*{rank}\left(  H_{m,n}\left(  x\right)  \right)  \leq
m$\textquotedblright\ as a \textquotedblleft global\textquotedblright%
\ statement about the matrix $H_{m,n}\left(  x\right)  $, whereas the
statements \textquotedblleft$v\ H_{m,n}\left(  x\right)  =0$\textquotedblright%
\ and \textquotedblleft$v\ H_{m-1,n+1}\left(  x\right)  =0$\textquotedblright%
\ are local in the sense that they only \textquotedblleft
sample\textquotedblright\ the matrix at a single vector each) and thus (as we
will soon see) are easier to sum.

\begin{proof}
[Proof of Proposition \ref{prop.elkies-prop2}.]We are in one of the following
two cases:

\textit{Case 1:} We have $\operatorname*{rank}\left(  H_{m,n}\left(  x\right)
\right)  >m$.

\textit{Case 2:} We have $\operatorname*{rank}\left(  H_{m,n}\left(  x\right)
\right)  \leq m$.

Let us first consider Case 1. In this case, we have $\operatorname*{rank}%
\left(  H_{m,n}\left(  x\right)  \right)  >m$. Thus, $\operatorname*{rank}%
\left(  H_{m,n}\left(  x\right)  \right)  =m+1$ (since $H_{m,n}\left(
x\right)  $ is an $\left(  m+1\right)  \times\left(  n+1\right)  $-matrix).
Therefore, the rows of the matrix $H_{m,n}\left(  x\right)  $ are linearly
independent. Hence, there exists no nonzero $v\in F^{1\times\left(
m+1\right)  }$ satisfying $v\ H_{m,n}\left(  x\right)  =0$. Therefore,%
\begin{equation}
\sum_{\substack{v\in F^{1\times\left(  m+1\right)  };\\v\neq0}}\left[
v\ H_{m,n}\left(  x\right)  =0\right]  =0. \label{pf.prop.elkies-prop2.IB.1}%
\end{equation}

Moreover, Lemma \ref{lem.hank-ranks.4} (applied to $m+n$, $m$ and $n+1$
instead of $n$, $p$ and $q$) yields that $\operatorname*{rank}\left(
H_{m-1,n+1}\left(  x\right)  \right)  =m$ (since $\operatorname*{rank}\left(
H_{m,n}\left(  x\right)  \right)  >m$). In other words, the $m\times\left(
n+2\right)  $-matrix $H_{m-1,n+1}\left(  x\right)  $ has rank $m$. Hence,
there exists no nonzero $v\in F^{1\times m}$ satisfying $v\ H_{m-1,n+1}\left(
x\right)  =0$. Therefore,%
\begin{equation}
\sum_{\substack{v\in F^{1\times m};\\v\neq0}}\left[  v\ H_{m-1,n+1}\left(
x\right)  =0\right]  =0. \label{pf.prop.elkies-prop2.IB.2}%
\end{equation}

Finally, $\left[  \operatorname*{rank}\left(  H_{m,n}\left(  x\right)
\right)  \leq m\right]  =0$ (since $\operatorname*{rank}\left(  H_{m,n}\left(
x\right)  \right)  >m$). In view of this equality, as well as
(\ref{pf.prop.elkies-prop2.IB.1}) and (\ref{pf.prop.elkies-prop2.IB.2}), the
equality that we are trying to prove rewrites as $\left(  q-1\right)
\cdot0=0-q\cdot0$, which is clearly true. Thus, Proposition
\ref{prop.elkies-prop2} is proved in Case 1.

Let us now consider Case 2. In this case, we have $\operatorname*{rank}\left(
H_{m,n}\left(  x\right)  \right)  \leq m$. Also, the matrix $H_{m-1,n+1}%
\left(  x\right)  $ has $m$ rows; thus, $\operatorname*{rank}\left(
H_{m-1,n+1}\left(  x\right)  \right)  \leq m\leq n+1$. Therefore, Lemma
\ref{lem.hank-ranks.3} (applied to $m+n$, $m$ and $n+1$ instead of $n$, $p$
and $q$) yields
\begin{equation}
\operatorname*{rank}\left(  H_{m,n}\left(  x\right)  \right)
=\operatorname*{rank}\left(  H_{m-1,n+1}\left(  x\right)  \right)  .
\label{pf.prop.elkies-prop2.IB.5}%
\end{equation}

Now, Proposition \ref{prop.iverson.rollcall} shows that $\sum_{v\in
F^{1\times\left(  m+1\right)  }}\left[  v\ H_{m,n}\left(  x\right)  =0\right]
$ is the number of all $v\in F^{1\times\left(  m+1\right)  }$ satisfying
$v\ H_{m,n}\left(  x\right)  =0$. In other words, $\sum_{v\in F^{1\times
\left(  m+1\right)  }}\left[  v\ H_{m,n}\left(  x\right)  =0\right]  $ is the
size of the left kernel\footnote{The \textit{left kernel} of an $s\times
t$-matrix $A\in F^{s\times t}$ is defined to be the set of all row vectors
$v\in F^{1\times s}$ satisfying $vA=0$. This is a vector subspace of
$F^{1\times s}$.} of the matrix $H_{m,n}\left(  x\right)  $. But the dimension
of this left kernel is $\left(  m+1\right)  -\operatorname*{rank}\left(
H_{m,n}\left(  x\right)  \right)  $ (by the rank-nullity theorem\footnote{The
\textit{rank-nullity theorem} (in the form we are using it here) says that the
dimension of the left kernel of a matrix $A\in F^{s\times t}$ equals
$s-\operatorname*{rank}A$.}); hence, the size of this left kernel is
$q^{\left(  m+1\right)  -\operatorname*{rank}\left(  H_{m,n}\left(  x\right)
\right)  }$. Thus,%
\begin{equation}
\sum_{v\in F^{1\times\left(  m+1\right)  }}\left[  v\ H_{m,n}\left(  x\right)
=0\right]  =q^{\left(  m+1\right)  -\operatorname*{rank}\left(  H_{m,n}\left(
x\right)  \right)  }. \label{pf.prop.elkies-prop2.IB.6a}%
\end{equation}
The same reasoning shows that%
\begin{equation}
\sum_{v\in F^{1\times m}}\left[  v\ H_{m-1,n+1}\left(  x\right)  =0\right]
=q^{m-\operatorname*{rank}\left(  H_{m-1,n+1}\left(  x\right)  \right)  }.
\label{pf.prop.elkies-prop2.IB.6b}%
\end{equation}

Now, (\ref{pf.prop.elkies-prop2.IB.5}) yields%
\[
q^{\left(  m+1\right)  -\operatorname*{rank}\left(  H_{m,n}\left(  x\right)
\right)  }=q^{\left(  m+1\right)  -\operatorname*{rank}\left(  H_{m-1,n+1}%
\left(  x\right)  \right)  }=q\cdot q^{m-\operatorname*{rank}\left(
H_{m-1,n+1}\left(  x\right)  \right)  }.
\]
In view of%
\begin{align*}
&  q^{\left(  m+1\right)  -\operatorname*{rank}\left(  H_{m,n}\left(
x\right)  \right)  }\\
&  =\sum_{v\in F^{1\times\left(  m+1\right)  }}\left[  v\ H_{m,n}\left(
x\right)  =0\right]  \ \ \ \ \ \ \ \ \ \ \left(  \text{by
(\ref{pf.prop.elkies-prop2.IB.6a})}\right) \\
&  =1+\sum_{\substack{v\in F^{1\times\left(  m+1\right)  };\\v\neq0}}\left[
v\ H_{m,n}\left(  x\right)  =0\right]  \ \ \ \ \ \ \ \ \ \ \left(  \text{since
}0\ H_{m,n}\left(  x\right)  =0\right)
\end{align*}
and%
\begin{align*}
&  q^{m-\operatorname*{rank}\left(  H_{m-1,n+1}\left(  x\right)  \right)  }\\
&  =\sum_{v\in F^{1\times m}}\left[  v\ H_{m-1,n+1}\left(  x\right)
=0\right]  \ \ \ \ \ \ \ \ \ \ \left(  \text{by
(\ref{pf.prop.elkies-prop2.IB.6b})}\right) \\
&  =1+\sum_{\substack{v\in F^{1\times m};\\v\neq0}}\left[  v\ H_{m-1,n+1}%
\left(  x\right)  =0\right]  \ \ \ \ \ \ \ \ \ \ \left(  \text{since
}0\ H_{m-1,n+1}\left(  x\right)  =0\right)  ,
\end{align*}
we can rewrite this as%
\[
1+\sum_{\substack{v\in F^{1\times\left(  m+1\right)  };\\v\neq0}}\left[
v\ H_{m,n}\left(  x\right)  =0\right]  =q\cdot\left(  1+\sum_{\substack{v\in
F^{1\times m};\\v\neq0}}\left[  v\ H_{m-1,n+1}\left(  x\right)  =0\right]
\right)  .
\]
In other words,%
\[
\sum_{\substack{v\in F^{1\times\left(  m+1\right)  };\\v\neq0}}\left[
v\ H_{m,n}\left(  x\right)  =0\right]  -q\sum_{\substack{v\in F^{1\times
m};\\v\neq0}}\left[  v\ H_{m-1,n+1}\left(  x\right)  =0\right]  =q-1.
\]
Comparing this with%
\[
\left(  q-1\right)  \cdot\underbrace{\left[  \operatorname*{rank}\left(
H_{m,n}\left(  x\right)  \right)  \leq m\right]  }_{\substack{=1\\\text{(since
}\operatorname*{rank}\left(  H_{m,n}\left(  x\right)  \right)  \leq m\text{)}%
}}=q-1,
\]
we obtain precisely the claim of Proposition \ref{prop.elkies-prop2}. Thus,
Proposition \ref{prop.elkies-prop2} is proved in Case 2.

We have now proved Proposition \ref{prop.elkies-prop2} in both Cases 1 and 2.
Thus, Proposition \ref{prop.elkies-prop2} always holds.
\end{proof}

\subsection{Sums over $x$ for fixed $v$}

We need another definition. Namely, if $m\in\mathbb{N}$, and if $v=\left(
v_{0},v_{1},\ldots,v_{m}\right)  \in F^{1\times\left(  m+1\right)  }$ is a row
vector of size $m+1$, then $\operatorname*{last}v$ will denote $v_{m}$ (that
is, the last entry of $v$). There is a bijection%
\begin{align*}
R:\left\{  v\in F^{1\times\left(  m+1\right)  }\mid\operatorname*{last}%
v=0\right\}   &  \rightarrow F^{1\times m},\\
\left(  v_{0},v_{1},\ldots,v_{m}\right)   &  \mapsto\left(  v_{0},v_{1}%
,\ldots,v_{m-1}\right)  .
\end{align*}
Its inverse map $R^{-1}$ sends each row vector $\left(  v_{0},v_{1}%
,\ldots,v_{m-1}\right)  \in F^{1\times m}$ to the row vector $\left(
v_{0},v_{1},\ldots,v_{m-1},0\right)  $.

\begin{lemma}
\label{lem.lastnon0}Let $k,m,n\in\mathbb{N}$ satisfy $k\leq m$. Let $v\in
F^{1\times\left(  m+1\right)  }$ be a row vector of size $m+1$ such that
$\operatorname*{last}v\neq0$. Fix any $k$-tuple $a\in F^{k}$. Then,%
\[
\sum_{\substack{x\in F^{m+n+1};\\x_{\left[  0,k\right)  }=a}}\left[
v\ H_{m,n}\left(  x\right)  =0\right]  =q^{m-k}.
\]

\end{lemma}

\begin{proof}
[Proof of Lemma \ref{lem.lastnon0}.]Proposition \ref{prop.iverson.rollcall}
shows that $\sum_{\substack{x\in F^{m+n+1};\\x_{\left[  0,k\right)  }%
=a}}\left[  v\ H_{m,n}\left(  x\right)  =0\right]  $ is the number of all
$x\in F^{m+n+1}$ satisfying $x_{\left[  0,k\right)  }=a$ and $v\ H_{m,n}%
\left(  x\right)  =0$. Thus, we must prove that this number is $q^{m-k}$.

Write $v$ and $a$ as $v=\left(  v_{0},v_{1},\ldots,v_{m}\right)  $ and
$a=\left(  a_{0},a_{1},\ldots,a_{k-1}\right)  $, respectively. Thus,
$\operatorname*{last}v=v_{m}$, so that $v_{m}=\operatorname*{last}v\neq0$.

Now, we are looking for an $x=\left(  x_{0},x_{1},\ldots,x_{m+n}\right)  \in
F^{m+n+1}$ satisfying $x_{\left[  0,k\right)  }=a$ and $v\ H_{m,n}\left(
x\right)  =0$. The condition $x_{\left[  0,k\right)  }=a$ says that the first
$k$ entries of $x$ equal the respective entries of $a$; that is, $x_{i}=a_{i}$
for each $i\in\left\{  0,1,\ldots,k-1\right\}  $. Thus, $x_{0},x_{1}%
,\ldots,x_{k-1}$ are uniquely determined. The condition $v\ H_{m,n}\left(
x\right)  =0$ is equivalent to $x$ satisfying the following system of linear
equations:%
\begin{equation}
\left\{
\begin{array}
[c]{c}%
v_{0}x_{0}+v_{1}x_{1}+\cdots+v_{m}x_{m}=0;\\
v_{0}x_{1}+v_{1}x_{2}+\cdots+v_{m}x_{m+1}=0;\\
v_{0}x_{2}+v_{1}x_{3}+\cdots+v_{m}x_{m+2}=0;\\
\cdots;\\
v_{0}x_{n}+v_{1}x_{n+1}+\cdots+v_{m}x_{m+n}=0.
\end{array}
\right.  \label{pf.lem.lastnon0.sys}%
\end{equation}
Since $v_{m}\neq0$, this latter system of equations can be uniquely solved for
the unknowns $x_{m},x_{m+1},\ldots,x_{m+n}$ (by recursive substitution) when
the $m$ entries $x_{0},x_{1},\ldots,x_{m-1}$ are given. Hence, each $\left(
m+n+1\right)  $-tuple $x=\left(  x_{0},x_{1},\ldots,x_{m+n}\right)  \in
F^{m+n+1}$ satisfying $x_{\left[  0,k\right)  }=a$ and $v\ H_{m,n}\left(
x\right)  =0$ can be constructed as follows:

\begin{itemize}
\item First, we set $x_{i}=a_{i}$ for each $i\in\left\{  0,1,\ldots
,k-1\right\}  $. This determines the first $k$ entries $x_{0},x_{1}%
,\ldots,x_{k-1}$ of $x$.

\item Then, we choose arbitrary values for the next $m-k$ entries
$x_{k},x_{k+1},\ldots,x_{m-1}$.

\item Finally, we uniquely determine the remaining entries $x_{m}%
,x_{m+1},\ldots,x_{m+n}$ by solving the system (\ref{pf.lem.lastnon0.sys}).
\end{itemize}

Clearly, the number of ways to perform this construction is $q^{m-k}$ (since
there are $\left\vert F\right\vert =q$ many options for each of the $m-k$
entries $x_{k},x_{k+1},\ldots,x_{m-1}$). Thus, the number of all $x\in
F^{m+n+1}$ satisfying $x_{\left[  0,k\right)  }=a$ and $v\ H_{m,n}\left(
x\right)  =0$ is $q^{m-k}$. This proves Lemma \ref{lem.lastnon0}.
\end{proof}

\begin{lemma}
\label{lem.last0}Let $k,m,n\in\mathbb{N}$ satisfy $k\leq n+1$. Let $v\in
F^{1\times\left(  m+1\right)  }$ be a nonzero row vector of size $m+1$ such
that $\operatorname*{last}v=0$. Fix any $k$-tuple $a\in F^{k}$. Then,%
\begin{align}
&  \sum_{\substack{x\in F^{m+n+1};\\x_{\left[  0,k\right)  }=a}}\left[
v\ H_{m,n}\left(  x\right)  =0\right] \nonumber\\
&  =q\sum_{\substack{x\in F^{m+n+1};\\x_{\left[  0,k\right)  }=a}}\left[
R\left(  v\right)  \ H_{m-1,n+1}\left(  x\right)  =0\right]  .
\label{eq.lem.last0.eq}%
\end{align}

\end{lemma}

\begin{proof}
[Proof of Lemma \ref{lem.last0}.]The sum on the left hand side of
(\ref{eq.lem.last0.eq}) is the number of all $\left(  m+n+1\right)  $-tuples
$x\in F^{m+n+1}$ satisfying $x_{\left[  0,k\right)  }=a$ and $v\ H_{m,n}%
\left(  x\right)  =0$ (because of Proposition \ref{prop.iverson.rollcall}).
Let us refer to such $\left(  m+n+1\right)  $-tuples $x$ as \textit{weakly
nice tuples}.

The sum on the right hand side of (\ref{eq.lem.last0.eq}) is the number of all
$\left(  m+n+1\right)  $-tuples $x\in F^{m+n+1}$ satisfying $x_{\left[
0,k\right)  }=a$ and $R\left(  v\right)  \ H_{m-1,n+1}\left(  x\right)  =0$
(because of Proposition \ref{prop.iverson.rollcall}). Let us refer to such
$\left(  m+n+1\right)  $-tuples $x$ as \textit{strongly nice tuples}.

We thus need to prove that the number of weakly nice tuples equals $q$ times
the number of strongly nice tuples.

We shall achieve this by constructing a bijection%
\[
\left\{  \text{weakly nice tuples}\right\}  \rightarrow F\times\left\{
\text{strongly nice tuples}\right\}  .
\]

Indeed, let us unravel the definitions of weakly and strongly nice tuples.

Write $v$ and $a$ as $v=\left(  v_{0},v_{1},\ldots,v_{m}\right)  $ and
$a=\left(  a_{0},a_{1},\ldots,a_{k-1}\right)  $, respectively. Thus,
$\operatorname*{last}v=v_{m}$, so that $v_{m}=\operatorname*{last}v=0$.
Consider the \textbf{largest} $j\in\left\{  0,1,\ldots,m\right\}  $ satisfying
$v_{j}\neq0$. (This exists, since $v$ is nonzero.) Thus, $v_{j}\neq0$ but
$v_{j+1}=v_{j+2}=\cdots=v_{m}=0$. Also, the definition of $R$ yields $R\left(
v\right)  =\left(  v_{0},v_{1},\ldots,v_{m-1}\right)  $.

We have $j\neq m$ (since $v_{j}\neq0$ but $v_{m}=0$). Thus, $j\leq m-1$.
Furthermore,
\[
\underbrace{j}_{\geq0}+n+1\geq n+1>n\geq k-1\ \ \ \ \ \ \ \ \ \ \left(
\text{since }k\leq n+1\right)  .
\]

The weakly nice tuples are the $\left(  m+n+1\right)  $-tuples $x=\left(
x_{0},x_{1},\ldots,x_{m+n}\right)  \in F^{m+n+1}$ satisfying%
\begin{equation}
x_{i}=a_{i}\ \ \ \ \ \ \ \ \ \ \text{for each }i\in\left\{  0,1,\ldots
,k-1\right\}  \label{pf.lem.last0.weakly.ai}%
\end{equation}
as well as%
\begin{equation}
\left\{
\begin{array}
[c]{c}%
v_{0}x_{0}+v_{1}x_{1}+\cdots+v_{m}x_{m}=0;\\
v_{0}x_{1}+v_{1}x_{2}+\cdots+v_{m}x_{m+1}=0;\\
v_{0}x_{2}+v_{1}x_{3}+\cdots+v_{m}x_{m+2}=0;\\
\cdots;\\
v_{0}x_{n}+v_{1}x_{n+1}+\cdots+v_{m}x_{m+n}=0
\end{array}
\right.  \label{pf.lem.last0.weakly.sys}%
\end{equation}
(because the condition \textquotedblleft$x_{\left[  0,k\right)  }%
=a$\textquotedblright\ is equivalent to (\ref{pf.lem.last0.weakly.ai}),
whereas the condition \textquotedblleft$v\ H_{m,n}\left(  x\right)
=0$\textquotedblright\ is equivalent to (\ref{pf.lem.last0.weakly.sys})). In
view of $v_{j+1}=v_{j+2}=\cdots=v_{m}=0$, we can rewrite this as follows: The
weakly nice tuples are the $\left(  m+n+1\right)  $-tuples $x=\left(
x_{0},x_{1},\ldots,x_{m+n}\right)  \in F^{m+n+1}$ satisfying%
\[
x_{i}=a_{i}\ \ \ \ \ \ \ \ \ \ \text{for each }i\in\left\{  0,1,\ldots
,k-1\right\}
\]
as well as%
\begin{equation}
\left\{
\begin{array}
[c]{c}%
v_{0}x_{0}+v_{1}x_{1}+\cdots+v_{j}x_{j}=0;\\
v_{0}x_{1}+v_{1}x_{2}+\cdots+v_{j}x_{j+1}=0;\\
v_{0}x_{2}+v_{1}x_{3}+\cdots+v_{j}x_{j+2}=0;\\
\cdots;\\
v_{0}x_{n}+v_{1}x_{n+1}+\cdots+v_{j}x_{j+n}=0.
\end{array}
\right.  \label{pf.lem.last0.sys1}%
\end{equation}

A similar argument (using $R\left(  v\right)  =\left(  v_{0},v_{1}%
,\ldots,v_{m-1}\right)  $) shows that the strongly nice tuples are the
$\left(  m+n+1\right)  $-tuples $x=\left(  x_{0},x_{1},\ldots,x_{m+n}\right)
\in F^{m+n+1}$ satisfying%
\[
x_{i}=a_{i}\ \ \ \ \ \ \ \ \ \ \text{for each }i\in\left\{  0,1,\ldots
,k-1\right\}
\]
as well as%
\begin{equation}
\left\{
\begin{array}
[c]{c}%
v_{0}x_{0}+v_{1}x_{1}+\cdots+v_{j}x_{j}=0;\\
v_{0}x_{1}+v_{1}x_{2}+\cdots+v_{j}x_{j+1}=0;\\
v_{0}x_{2}+v_{1}x_{3}+\cdots+v_{j}x_{j+2}=0;\\
\cdots;\\
v_{0}x_{n}+v_{1}x_{n+1}+\cdots+v_{j}x_{j+n}=0;\\
v_{0}x_{n+1}+v_{1}x_{n+2}+\cdots+v_{j}x_{j+n+1}=0.
\end{array}
\right.  \label{pf.lem.last0.sys2}%
\end{equation}

These characterizations of weakly and strongly nice tuples are very similar:
The system (\ref{pf.lem.last0.sys2}) consists of all the equations of
(\ref{pf.lem.last0.sys1}) as well as one extra equation
\begin{equation}
v_{0}x_{n+1}+v_{1}x_{n+2}+\cdots+v_{j}x_{j+n+1}=0.
\label{pf.lem.last0.extra-eqn}%
\end{equation}
This latter equation (\ref{pf.lem.last0.extra-eqn}) uniquely determines the
entry $x_{j+n+1}$ in terms of the other entries of $x$ (since $v_{j}\neq0$),
whereas $x_{j+n+1}$ is entirely unconstrained by the system
(\ref{pf.lem.last0.sys1}). Thus, the entry $x_{j+n+1}$ is uniquely determined
(in terms of the other entries of $x$) in a strongly nice tuple $x$, while
being entirely unconstrained in a weakly nice tuple\footnote{Here we are using
the fact that the \textquotedblleft$x_{i}=a_{i}$ for each $i\in\left\{
0,1,\ldots,k-1\right\}  $\textquotedblright\ conditions don't constrain
$x_{j+n+1}$ either (since $j+n+1>k-1$).}. Informally speaking, this shows that
a weakly nice tuple has \textquotedblleft one more degree of
freedom\textquotedblright\ than a strongly nice tuple (and this degree of
freedom is the entry $x_{j+n+1}$, which can take $q$ possible values in a
weakly nice tuple). This easily entails that the number of weakly nice tuples
equals $q$ times the number of strongly nice tuples\footnote{Here is a
rigorous way to show this: Consider the map%
\begin{align*}
\alpha:F\times\left\{  \text{strongly nice tuples}\right\}   &  \rightarrow
\left\{  \text{weakly nice tuples}\right\}  ,\\
\left(  y,\left(  x_{0},x_{1},\ldots,x_{m+n}\right)  \right)   &
\mapsto\left(  x_{0},x_{1},\ldots,x_{j+n},y,x_{j+n+2},x_{j+n+3},\ldots
,x_{m+n}\right)  ,
\end{align*}
which simply replaces the entry $x_{j+n+1}$ of the strongly nice tuple
$\left(  x_{0},x_{1},\ldots,x_{m+n}\right)  $ by the element $y$. Consider the
map
\begin{align*}
\beta:\left\{  \text{weakly nice tuples}\right\}   &  \rightarrow
F\times\left\{  \text{strongly nice tuples}\right\}  ,\\
\left(  x_{0},x_{1},\ldots,x_{m+n}\right)   &  \mapsto\left(  x_{j+n+1}%
,\left(  x_{0},x_{1},\ldots,x_{j+n},z,x_{j+n+2},x_{j+n+3},\ldots
,x_{m+n}\right)  \right)  ,
\end{align*}
where $z$ is the unique element of $F$ that would make the equation
(\ref{pf.lem.last0.extra-eqn}) valid when it is substituted for $x_{j+n+1}$
(that is, explicitly, $z$ is given by the formula $z=-\left(  v_{0}%
x_{n+1}+v_{1}x_{n+2}+\cdots+v_{j-1}x_{j+n}\right)  /v_{j}$). Our above
characterizations of weakly nice and strongly nice tuples show that these two
maps $\alpha$ and $\beta$ are mutually inverse. Hence, $\alpha$ and $\beta$
are bijections. Thus,%
\begin{align*}
\left\vert \left\{  \text{weakly nice tuples}\right\}  \right\vert  &
=\left\vert F\times\left\{  \text{strongly nice tuples}\right\}  \right\vert
\\
&  =q\cdot\left\vert \left\{  \text{strongly nice tuples}\right\}  \right\vert
.
\end{align*}
In other words, the number of weakly nice tuples equals $q$ times the number
of strongly nice tuples.}. This proves Lemma \ref{lem.last0}.
\end{proof}

Lemma \ref{lem.last0} and Lemma \ref{lem.lastnon0} combined lead to the following:

\begin{lemma}
\label{lem.sumlast}Let $k,m,n\in\mathbb{N}$ satisfy $k\leq m$ and $k\leq n+1$.
Fix any $k$-tuple $a\in F^{k}$. Then,%
\begin{align*}
&  \sum_{\substack{x\in F^{m+n+1};\\x_{\left[  0,k\right)  }=a}}\ \ \sum
_{\substack{v\in F^{1\times\left(  m+1\right)  };\\v\neq0}}\left[
v\ H_{m,n}\left(  x\right)  =0\right]  -q\sum_{\substack{x\in F^{m+n+1}%
;\\x_{\left[  0,k\right)  }=a}}\ \ \sum_{\substack{v\in F^{1\times m}%
;\\v\neq0}}\left[  v\ H_{m-1,n+1}\left(  x\right)  =0\right] \\
&  =\left(  q-1\right)  q^{2m-k}.
\end{align*}

\end{lemma}

\begin{proof}
[Proof of Lemma \ref{lem.sumlast}.]We first observe that%
\begin{align}
&  \left(  \text{the number of all vectors }v\in F^{1\times\left(  m+1\right)
}\text{ satisfying }\operatorname*{last}v\neq0\right)  \nonumber\\
&  =\left(  q-1\right)  q^{m}\label{pf.lem.sumlast.num-last-nonz}%
\end{align}
(since a vector $v\in F^{1\times\left(  m+1\right)  }$ satisfying
$\operatorname*{last}v\neq0$ can be constructed by choosing its last entry
from the $\left(  q-1\right)  $-element set $F\setminus\left\{  0\right\}  $
and then choosing its remaining $m$ entries from the $q$-element set $F$).

For any row vector $v\in F^{1\times\left(  m+1\right)  }$, we define a number%
\begin{equation}
\chi_{v}:=\sum_{\substack{x\in F^{m+n+1};\\x_{\left[  0,k\right)  }=a}}\left[
v\ H_{m,n}\left(  x\right)  =0\right]  .\label{pf.lem.sumlast.chiv=}%
\end{equation}
Thus, if $v\in F^{1\times\left(  m+1\right)  }$ is a row vector satisfying
$\operatorname*{last}v\neq0$, then%
\begin{equation}
\chi_{v}=\sum_{\substack{x\in F^{m+n+1};\\x_{\left[  0,k\right)  }=a}}\left[
v\ H_{m,n}\left(  x\right)  =0\right]  =q^{m-k}%
\label{pf.lem.sumlast.chiv=qm-k}%
\end{equation}
(by Lemma \ref{lem.lastnon0}).

Recall that $R:\left\{  v\in F^{1\times\left(  m+1\right)  }\mid
\operatorname*{last}v=0\right\}  \rightarrow F^{1\times m}$ is a bijection.
This bijection sends $0$ to $0$, and therefore restricts to a bijection%
\begin{align*}
\left\{  v\in F^{1\times\left(  m+1\right)  }\mid\operatorname*{last}v=0\text{
and }v\neq0\right\}   &  \rightarrow\left\{  v\in F^{1\times m}\mid
v\neq0\right\}  ,\\
v &  \mapsto R\left(  v\right)  .
\end{align*}
Hence, given any $x\in F^{m+n+1}$, we can substitute $R\left(  v\right)  $ for
$v$ in the sum \newline$\sum_{\substack{v\in F^{1\times m};\\v\neq0}}\left[
v\ H_{m-1,n+1}\left(  x\right)  =0\right]  $, and thus obtain%
\[
\sum_{\substack{v\in F^{1\times m};\\v\neq0}}\left[  v\ H_{m-1,n+1}\left(
x\right)  =0\right]  =\sum_{\substack{v\in F^{1\times\left(  m+1\right)
};\\\operatorname*{last}v=0;\\v\neq0}}\left[  R\left(  v\right)
\ H_{m-1,n+1}\left(  x\right)  =0\right]  .
\]
Thus,%
\begin{align}
& q\sum_{\substack{x\in F^{m+n+1};\\x_{\left[  0,k\right)  }=a}}\ \ \sum
_{\substack{v\in F^{1\times m};\\v\neq0}}\left[  v\ H_{m-1,n+1}\left(
x\right)  =0\right]  \nonumber\\
& =q\sum_{\substack{x\in F^{m+n+1};\\x_{\left[  0,k\right)  }=a}%
}\ \ \sum_{\substack{v\in F^{1\times\left(  m+1\right)  }%
;\\\operatorname*{last}v=0;\\v\neq0}}\left[  R\left(  v\right)  \ H_{m-1,n+1}%
\left(  x\right)  =0\right]  \nonumber\\
& =\sum_{\substack{v\in F^{1\times\left(  m+1\right)  };\\\operatorname*{last}%
v=0;\\v\neq0}}\underbrace{q\sum_{\substack{x\in F^{m+n+1};\\x_{\left[
0,k\right)  }=a}}\left[  R\left(  v\right)  \ H_{m-1,n+1}\left(  x\right)
=0\right]  }_{\substack{=\sum_{\substack{x\in F^{m+n+1};\\x_{\left[
0,k\right)  }=a}}\left[  v\ H_{m,n}\left(  x\right)  =0\right]  \\\text{(by
Lemma \ref{lem.last0})}}}\nonumber\\
& =\sum_{\substack{v\in F^{1\times\left(  m+1\right)  };\\\operatorname*{last}%
v=0;\\v\neq0}}\ \ \underbrace{\sum_{\substack{x\in F^{m+n+1};\\x_{\left[
0,k\right)  }=a}}\left[  v\ H_{m,n}\left(  x\right)  =0\right]  }%
_{\substack{=\chi_{v}\\\text{(by (\ref{pf.lem.sumlast.chiv=}))}}}\nonumber\\
& =\sum_{\substack{v\in F^{1\times\left(  m+1\right)  };\\\operatorname*{last}%
v=0;\\v\neq0}}\chi_{v}=\sum_{\substack{v\in F^{1\times\left(  m+1\right)
};\\v\neq0;\\\operatorname*{last}v=0}}\chi_{v}.\label{pf.lem.sumlast.2}%
\end{align}
On the other hand,%
\begin{align}
& \sum_{\substack{x\in F^{m+n+1};\\x_{\left[  0,k\right)  }=a}}\ \ \sum
_{\substack{v\in F^{1\times\left(  m+1\right)  };\\v\neq0}}\left[
v\ H_{m,n}\left(  x\right)  =0\right]  \nonumber\\
& =\sum_{\substack{v\in F^{1\times\left(  m+1\right)  };\\v\neq0}%
}\ \ \underbrace{\sum_{\substack{x\in F^{m+n+1};\\x_{\left[  0,k\right)  }%
=a}}\left[  v\ H_{m,n}\left(  x\right)  =0\right]  }_{\substack{=\chi
_{v}\\\text{(by (\ref{pf.lem.sumlast.chiv=}))}}}\nonumber\\
& =\sum_{\substack{v\in F^{1\times\left(  m+1\right)  };\\v\neq0}}\chi
_{v}.\label{pf.lem.sumlast.1}%
\end{align}
Subtracting the equality (\ref{pf.lem.sumlast.2}) from the equality
(\ref{pf.lem.sumlast.1}), we obtain
\begin{align*}
&  \sum_{\substack{x\in F^{m+n+1};\\x_{\left[  0,k\right)  }=a}}\ \ \sum
_{\substack{v\in F^{1\times\left(  m+1\right)  };\\v\neq0}}\left[
v\ H_{m,n}\left(  x\right)  =0\right]  -q\sum_{\substack{x\in F^{m+n+1}%
;\\x_{\left[  0,k\right)  }=a}}\ \ \sum_{\substack{v\in F^{1\times m}%
;\\v\neq0}}\left[  v\ H_{m-1,n+1}\left(  x\right)  =0\right]  \\
&  =\sum_{\substack{v\in F^{1\times\left(  m+1\right)  };\\v\neq0}}\chi
_{v}-\sum_{\substack{v\in F^{1\times\left(  m+1\right)  };\\v\neq
0;\\\operatorname*{last}v=0}}\chi_{v}=\sum_{\substack{v\in F^{1\times\left(
m+1\right)  };\\v\neq0;\\\operatorname*{last}v\neq0}}\ \ \underbrace{\chi_{v}%
}_{\substack{=q^{m-k}\\\text{(by (\ref{pf.lem.sumlast.chiv=qm-k}))}}}\\
&  \ \ \ \ \ \ \ \ \ \ \ \ \ \ \ \ \ \ \ \ \left(
\begin{array}
[c]{c}%
\text{since }\sum_{\substack{v\in F^{1\times\left(  m+1\right)  };\\v\neq
0}}\rho_{v}-\sum_{\substack{v\in F^{1\times\left(  m+1\right)  }%
;\\v\neq0;\\\operatorname*{last}v=0}}\rho_{v}=\sum_{\substack{v\in
F^{1\times\left(  m+1\right)  };\\v\neq0;\\\operatorname*{last}v\neq0}%
}\rho_{v}\\
\text{for any numbers }\rho_{v}%
\end{array}
\right)  \\
&  =\sum_{\substack{v\in F^{1\times\left(  m+1\right)  };\\v\neq
0;\\\operatorname*{last}v\neq0}}q^{m-k}=\sum_{\substack{v\in F^{1\times\left(
m+1\right)  };\\\operatorname*{last}v\neq0}}q^{m-k}\\
&  \ \ \ \ \ \ \ \ \ \ \ \ \ \ \ \ \ \ \ \ \left(
\begin{array}
[c]{c}%
\text{here, we have removed the condition \textquotedblleft}v\neq
0\text{\textquotedblright\ from under}\\
\text{the summation sign, since any vector }v\in F^{1\times\left(  m+1\right)
}\\
\text{satisfying }\operatorname*{last}v\neq0\text{ automatically satisfies
}v\neq0
\end{array}
\right)  \\
&  =\underbrace{\left(  \text{the number of all vectors }v\in F^{1\times
\left(  m+1\right)  }\text{ satisfying }\operatorname*{last}v\neq0\right)
}_{\substack{=\left(  q-1\right)  q^{m}\\\text{(by
(\ref{pf.lem.sumlast.num-last-nonz}))}}}\cdot q^{m-k}\\
&  =\left(  q-1\right)  \underbrace{q^{m}\cdot q^{m-k}}_{=q^{2m-k}}=\left(
q-1\right)  q^{2m-k}.
\end{align*}
This proves Lemma \ref{lem.sumlast}.
\end{proof}

\subsection{Theorem \ref{thm.hankel.mn<=rk} for $r=m$}

Before we prove Theorem \ref{thm.hankel.mn<=rk} in full generality, let us
first show it in the particular case when $r=m$:

\begin{lemma}
\label{lem.hankel.case-r=m}Let $k,m,n\in\mathbb{N}$ satisfy $k\leq m\leq n+1$.
Fix any $k$-tuple $a\in F^{k}$. The number of $\left(  m+n+1\right)  $-tuples
$x\in F^{m+n+1}$ satisfying $x_{\left[  0,k\right)  }=a$ and
$\operatorname*{rank}\left(  H_{m,n}\left(  x\right)  \right)  \leq m$ is
$q^{2m-k}$.
\end{lemma}

\begin{proof}
[Proof of Lemma \ref{lem.hankel.case-r=m}.]Write the $k$-tuple $a$ in the form
$a=\left(  a_{0},a_{1},\ldots,a_{k-1}\right)  $.

We have%
\begin{align*}
&  \left(  q-1\right)  \cdot\sum_{\substack{x\in F^{m+n+1};\\x_{\left[
0,k\right)  }=a}}\left[  \operatorname*{rank}\left(  H_{m,n}\left(  x\right)
\right)  \leq m\right] \\
&  =\sum_{\substack{x\in F^{m+n+1};\\x_{\left[  0,k\right)  }=a}%
}\ \ \ \ \underbrace{\left(  q-1\right)  \cdot\left[  \operatorname*{rank}%
\left(  H_{m,n}\left(  x\right)  \right)  \leq m\right]  }_{\substack{=\sum
_{\substack{v\in F^{1\times\left(  m+1\right)  };\\v\neq0}}\left[
v\ H_{m,n}\left(  x\right)  =0\right]  -q\sum_{\substack{v\in F^{1\times
m};\\v\neq0}}\left[  v\ H_{m-1,n+1}\left(  x\right)  =0\right]  \\\text{(by
Proposition \ref{prop.elkies-prop2})}}}\\
&  =\sum_{\substack{x\in F^{m+n+1};\\x_{\left[  0,k\right)  }=a}}\left(
\sum_{\substack{v\in F^{1\times\left(  m+1\right)  };\\v\neq0}}\left[
v\ H_{m,n}\left(  x\right)  =0\right]  -q\sum_{\substack{v\in F^{1\times
m};\\v\neq0}}\left[  v\ H_{m-1,n+1}\left(  x\right)  =0\right]  \right) \\
&  =\sum_{\substack{x\in F^{m+n+1};\\x_{\left[  0,k\right)  }=a}%
}\ \ \sum_{\substack{v\in F^{1\times\left(  m+1\right)  };\\v\neq0}}\left[
v\ H_{m,n}\left(  x\right)  =0\right]  -q\sum_{\substack{x\in F^{m+n+1}%
;\\x_{\left[  0,k\right)  }=a}}\ \ \sum_{\substack{v\in F^{1\times m}%
;\\v\neq0}}\left[  v\ H_{m-1,n+1}\left(  x\right)  =0\right] \\
&  =\left(  q-1\right)  q^{2m-k}\ \ \ \ \ \ \ \ \ \ \left(  \text{by Lemma
\ref{lem.sumlast}}\right)  .
\end{align*}
Cancelling $q-1$ from this equality (since $q-1\neq0$), we obtain
\[
\sum_{\substack{x\in F^{m+n+1};\\x_{\left[  0,k\right)  }=a}}\left[
\operatorname*{rank}\left(  H_{m,n}\left(  x\right)  \right)  \leq m\right]
=q^{2m-k}.
\]
But the left hand side of this equality is the number of $\left(
m+n+1\right)  $-tuples $x\in F^{m+n+1}$ satisfying $x_{\left[  0,k\right)
}=a$ and $\operatorname*{rank}\left(  H_{m,n}\left(  x\right)  \right)  \leq
m$ (because of Proposition \ref{prop.iverson.rollcall}). Thus, this number is
$q^{2m-k}$. This proves Lemma \ref{lem.hankel.case-r=m}.
\end{proof}

\section{\label{sec.proofs-main}Proofs of the main results}

We can now prove the results from Section \ref{sec.results} in their full generality.

\begin{proof}
[Proof of Theorem \ref{thm.hankel.mn<=rk}.]Let $s=m+n-r$. Then, $r+s=m+n$.
Also, $r\leq s$ (since $\underbrace{s}_{=m+n-r}-r=m+n-\underbrace{r}_{\leq
m}-\underbrace{r}_{\leq n}\geq m+n-m-n=0$), so that $r\leq s\leq s+1$ and thus
$k\leq r\leq s+1$.

Lemma \ref{lem.hank-ranks.6} (applied to $u=m+n$) yields the logical
equivalence
\[
\left(  \operatorname*{rank}\left(  H_{m,n}\left(  x\right)  \right)  \leq
r\right)  \ \Longleftrightarrow\ \left(  \operatorname*{rank}\left(
H_{r,s}\left(  x\right)  \right)  \leq r\right)
\]
for any $\left(  m+n+1\right)  $-tuple $x\in F^{m+n+1}$. Thus, \footnote{The
symbol \textquotedblleft\#\textquotedblright\ means \textquotedblleft
number\textquotedblright.}
\begin{align*}
&  \left(  \text{\# of all }\left(  m+n+1\right)  \text{-tuples }x\in
F^{m+n+1}\text{ satisfying }x_{\left[  0,k\right)  }=a\right. \\
&  \ \ \ \ \ \ \ \ \ \ \left.  \text{and }\operatorname*{rank}\left(
H_{m,n}\left(  x\right)  \right)  \leq r\vphantom{F^{m+n+1}}\right) \\
&  =\left(  \text{\# of all }\left(  m+n+1\right)  \text{-tuples }x\in
F^{m+n+1}\text{ satisfying }x_{\left[  0,k\right)  }=a\right. \\
&  \ \ \ \ \ \ \ \ \ \ \left.  \text{and }\operatorname*{rank}\left(
H_{r,s}\left(  x\right)  \right)  \leq r\vphantom{F^{m+n+1}}\right) \\
&  =\left(  \text{\# of all }\left(  r+s+1\right)  \text{-tuples }x\in
F^{r+s+1}\text{ satisfying }x_{\left[  0,k\right)  }=a\right. \\
&  \ \ \ \ \ \ \ \ \ \ \left.  \text{and }\operatorname*{rank}\left(
H_{r,s}\left(  x\right)  \right)  \leq r\vphantom{F^{m+n+1}}\right)
\ \ \ \ \ \ \ \ \ \ \left(  \text{since }m+n=r+s\right) \\
&  =q^{2r-k}\ \ \ \ \ \ \ \ \ \ \left(  \text{by Lemma
\ref{lem.hankel.case-r=m}, applied to }r\text{ and }s\text{ instead of
}m\text{ and }n\right)  .
\end{align*}
This proves Theorem \ref{thm.hankel.mn<=rk}.
\end{proof}

\begin{proof}
[Proof of Theorem \ref{thm.hankel.mn<=r}.]Let $a$ be the $0$-tuple $\left(
{}\right)  \in F^{0}$. Thus, Theorem \ref{thm.hankel.mn<=rk} (applied to
$k=0$) yields that the number of $\left(  m+n+1\right)  $-tuples $x\in
F^{m+n+1}$ satisfying $x_{\left[  0,0\right)  }=a$ and $\operatorname*{rank}%
\left(  H_{m,n}\left(  x\right)  \right)  \leq r$ is $q^{2r-0}$. We can remove
the \textquotedblleft$x_{\left[  0,0\right)  }=a$\textquotedblright\ condition
from the previous sentence (since \textbf{every} $\left(  m+n+1\right)
$-tuple $x\in F^{m+n+1}$ satisfies $x_{\left[  0,0\right)  }=\left(
{}\right)  =a$), and thus obtain the following: The number of $\left(
m+n+1\right)  $-tuples $x\in F^{m+n+1}$ satisfying $\operatorname*{rank}%
\left(  H_{m,n}\left(  x\right)  \right)  \leq r$ is $q^{2r-0}$. But this is
precisely the claim of Theorem \ref{thm.hankel.mn<=r} (since $2r-0=2r$). Thus,
Theorem \ref{thm.hankel.mn<=r} is proved.
\end{proof}

\begin{proof}
[Proof of Corollary \ref{cor.hankel.mn=r}.]We need to prove the following four
claims: \footnote{The symbol \textquotedblleft\#\textquotedblright\ means
\textquotedblleft number\textquotedblright.}

\begin{statement}
\textit{Claim 1:} If $r=0$, then the \# of $\left(  m+n+1\right)  $-tuples
$x\in F^{m+n+1}$ satisfying $\operatorname*{rank}\left(  H_{m,n}\left(
x\right)  \right)  =r$ is $1$.
\end{statement}

\begin{statement}
\textit{Claim 2:} If $0<r\leq m$, then the \# of $\left(  m+n+1\right)
$-tuples $x\in F^{m+n+1}$ satisfying $\operatorname*{rank}\left(
H_{m,n}\left(  x\right)  \right)  =r$ is $q^{2r-2}\left(  q^{2}-1\right)  $.
\end{statement}

\begin{statement}
\textit{Claim 3:} If $r=m+1$, then the \# of $\left(  m+n+1\right)  $-tuples
$x\in F^{m+n+1}$ satisfying $\operatorname*{rank}\left(  H_{m,n}\left(
x\right)  \right)  =r$ is $q^{2r-2}\left(  q^{n-m+1}-1\right)  $.
\end{statement}

\begin{statement}
\textit{Claim 4:} If $r>m+1$, then the \# of $\left(  m+n+1\right)  $-tuples
$x\in F^{m+n+1}$ satisfying $\operatorname*{rank}\left(  H_{m,n}\left(
x\right)  \right)  =r$ is $0$.
\end{statement}

[\textit{Proof of Claim 1:} We need to show that the \# of $\left(
m+n+1\right)  $-tuples $x\in F^{m+n+1}$ satisfying $\operatorname*{rank}%
\left(  H_{m,n}\left(  x\right)  \right)  =0$ is $1$. In other words, we need
to show that there is exactly one $\left(  m+n+1\right)  $-tuple $x\in
F^{m+n+1}$ satisfying $\operatorname*{rank}\left(  H_{m,n}\left(  x\right)
\right)  =0$. But this is rather simple: The $\left(  m+n+1\right)  $-tuple
$\left(  0,0,\ldots,0\right)  \in F^{m+n+1}$ does satisfy
$\operatorname*{rank}\left(  H_{m,n}\left(  x\right)  \right)  =0$ (since
$H_{m,n}\left(  x\right)  $ is the zero matrix when $x$ is this $\left(
m+n+1\right)  $-tuple), and no other $\left(  m+n+1\right)  $-tuple does this
(because if $x\in F^{m+n+1}$ is not $\left(  0,0,\ldots,0\right)  $, then the
matrix $H_{m,n}\left(  x\right)  $ has at least one nonzero entry, and
therefore its rank cannot be $0$). Thus, Claim 1 is proved.]

[\textit{Proof of Claim 2:} Assume that $0<r\leq m$. Thus, $r$ and $r-1$ are
elements of $\mathbb{N}$ and satisfy $r\leq m\leq n$ and $r-1\leq r\leq m\leq
n$. Hence:

\begin{itemize}
\item Theorem \ref{thm.hankel.mn<=r} yields that%
\begin{align*}
&  \left(  \text{\# of }\left(  m+n+1\right)  \text{-tuples }x\in
F^{m+n+1}\text{ satisfying }\operatorname*{rank}\left(  H_{m,n}\left(
x\right)  \right)  \leq r\right) \\
&  =q^{2r}.
\end{align*}

\item Theorem \ref{thm.hankel.mn<=r} (applied to $r-1$ instead of $r$) yields
that%
\begin{align*}
&  \left(  \text{\# of }\left(  m+n+1\right)  \text{-tuples }x\in
F^{m+n+1}\text{ satisfying }\operatorname*{rank}\left(  H_{m,n}\left(
x\right)  \right)  \leq r-1\right) \\
&  =q^{2\left(  r-1\right)  }.
\end{align*}

\end{itemize}

However, a matrix $A$ satisfies $\operatorname*{rank}A=r$ if and only if it
satisfies $\operatorname*{rank}A\leq r$ but not $\operatorname*{rank}A\leq
r-1$. Hence,%
\begin{align*}
&  \left(  \text{\# of }\left(  m+n+1\right)  \text{-tuples }x\in
F^{m+n+1}\text{ satisfying }\operatorname*{rank}\left(  H_{m,n}\left(
x\right)  \right)  =r\right) \\
&  =\underbrace{\left(  \text{\# of }\left(  m+n+1\right)  \text{-tuples }x\in
F^{m+n+1}\text{ satisfying }\operatorname*{rank}\left(  H_{m,n}\left(
x\right)  \right)  \leq r\right)  }_{=q^{2r}}\\
&  \ \ \ \ \ \ \ \ \ \ -\underbrace{\left(  \text{\# of }\left(  m+n+1\right)
\text{-tuples }x\in F^{m+n+1}\text{ satisfying }\operatorname*{rank}\left(
H_{m,n}\left(  x\right)  \right)  \leq r-1\right)  }_{=q^{2\left(  r-1\right)
}}\\
&  =q^{2r}-q^{2\left(  r-1\right)  }=q^{2r-2}\left(  q^{2}-1\right)  .
\end{align*}
This proves Claim 2.]

[\textit{Proof of Claim 3:} Assume that $r=m+1$. Thus, $2r=2\left(
m+1\right)  =2m+2$, so that $2m=2r-2$. The matrix $H_{m,n}\left(  x\right)  $
(for any given $x$) is an $\left(  m+1\right)  \times\left(  n+1\right)
$-matrix; thus, its rank is always $\leq m+1$. Hence, it has rank $m+1$ if and
only if it does not have rank $\leq m$. Thus,%
\begin{align*}
&  \left(  \text{\# of }\left(  m+n+1\right)  \text{-tuples }x\in
F^{m+n+1}\text{ satisfying }\operatorname*{rank}\left(  H_{m,n}\left(
x\right)  \right)  =m+1\right) \\
&  =\underbrace{\left(  \text{\# of all }\left(  m+n+1\right)  \text{-tuples
}x\in F^{m+n+1}\right)  }_{\substack{=q^{m+n+1}\\\text{(since }\left\vert
F\right\vert =q\text{)}}}\\
&  \ \ \ \ \ \ \ \ \ \ -\underbrace{\left(  \text{\# of }\left(  m+n+1\right)
\text{-tuples }x\in F^{m+n+1}\text{ satisfying }\operatorname*{rank}\left(
H_{m,n}\left(  x\right)  \right)  \leq m\right)  }_{\substack{=q^{2m}%
\\\text{(by Theorem \ref{thm.hankel.mn<=r}, applied to }m\text{ instead of
}r\text{)}}}\\
&  =q^{m+n+1}-q^{2m}=q^{2m}\left(  q^{n-m+1}-1\right)  =q^{2r-2}\left(
q^{n-m+1}-1\right)
\end{align*}
(since $2m=2r-2$). But this is precisely the claim of Claim 3 (since $r=m+1$).
Thus, Claim 3 is proven.]

[\textit{Proof of Claim 4:} Assume that $r>m+1$. The matrix $H_{m,n}\left(
x\right)  $ (for any given $x$) is an $\left(  m+1\right)  \times\left(
n+1\right)  $-matrix; thus, its rank is always $\leq m+1$. Hence, its rank is
never $r$ (because $r>m+1$). Thus,
\[
\left(  \text{\# of }\left(  m+n+1\right)  \text{-tuples }x\in F^{m+n+1}\text{
satisfying }\operatorname*{rank}\left(  H_{m,n}\left(  x\right)  \right)
=r\right)  =0.
\]
This proves Claim 4.]

Having proved all four claims, we thus have completed the proof of Corollary
\ref{cor.hankel.mn=r}.
\end{proof}

\begin{proof}
[Proof of Corollary \ref{cor.hankel.detk}.]If $x\in F^{2n+1}$ is any $\left(
2n+1\right)  $-tuple, then the condition \textquotedblleft$\det\left(
H_{n,n}\left(  x\right)  \right)  =0$\textquotedblright\ is equivalent to
\textquotedblleft$\operatorname*{rank}\left(  H_{n,n}\left(  x\right)
\right)  \leq n$\textquotedblright\ (since $H_{n,n}\left(  x\right)  $ is an
$\left(  n+1\right)  \times\left(  n+1\right)  $-matrix, and thus its
determinant vanishes if and only if its rank is $\leq n$). Hence, the number
of $\left(  2n+1\right)  $-tuples $x\in F^{2n+1}$ satisfying $x_{\left[
0,k\right)  }=a$ and $\det\left(  H_{n,n}\left(  x\right)  \right)  =0$ is
precisely the number of $\left(  2n+1\right)  $-tuples $x\in F^{2n+1}$
satisfying $x_{\left[  0,k\right)  }=a$ and $\operatorname*{rank}\left(
H_{n,n}\left(  x\right)  \right)  \leq n$. But Theorem \ref{thm.hankel.mn<=rk}
(applied to $m=n$ and $r=n$) shows that the latter number is $q^{2n-k}$. This
proves Corollary \ref{cor.hankel.detk}.
\end{proof}

\section{\label{sec.jt}Application to Jacobi--Trudi matrices}

Let us now discuss how \cite[Corollary 6.4]{Anzis-etc} follows from Corollary
\ref{cor.hankel.detk}. For the sake of simplicity, we shall first restate
\cite[Corollary 6.4]{Anzis-etc} in a self-contained form that does not rely on
the concepts of symmetric functions:

\begin{corollary}
\label{cor.anzis}Assume that $F$ is finite. Let $q=\left\vert F\right\vert $.
Let $u,v\in\mathbb{N}$. For each $\left(  u+v-1\right)  $-tuple $y=\left(
y_{1},y_{2},\ldots,y_{u+v-1}\right)  \in F^{u+v-1}$, we define the matrix%
\[
J_{u,v}\left(  y\right)  :=\left(  y_{u-i+j}\right)  _{1\leq i\leq v,\ 1\leq
j\leq v}\in F^{v\times v},
\]
where we set $y_{0}:=1$ and $y_{k}:=0$ for all $k<0$.

Then, the number of all $\left(  u+v-1\right)  $-tuples $y\in F^{u+v-1}$
satisfying $\det\left(  J_{u,v}\left(  y\right)  \right)  =0$ is $q^{u+v-2}$.
\end{corollary}

\begin{example}
\textbf{(a)} If $u=1$ and $v=3$, then each $3$-tuple $y=\left(  y_{1}%
,y_{2},y_{3}\right)  \in F^{3}$ satisfies%
\begin{align*}
J_{u,v}\left(  y\right)   &  =J_{1,3}\left(  y\right)  =\left(  y_{1-i+j}%
\right)  _{1\leq i\leq3,\ 1\leq j\leq3}=\left(
\begin{array}
[c]{ccc}%
y_{1} & y_{2} & y_{3}\\
y_{0} & y_{1} & y_{2}\\
y_{-1} & y_{0} & y_{1}%
\end{array}
\right) \\
&  =\left(
\begin{array}
[c]{ccc}%
y_{1} & y_{2} & y_{3}\\
1 & y_{1} & y_{2}\\
0 & 1 & y_{1}%
\end{array}
\right)  \ \ \ \ \ \ \ \ \ \ \left(  \text{since }y_{0}=1\text{ and }%
y_{-1}=0\right)
\end{align*}
and thus $\det\left(  J_{u,v}\left(  y\right)  \right)  =y_{3}+y_{1}%
^{3}-2y_{1}y_{2}$.

\textbf{(b)} If $u=4$ and $v=3$, then each $6$-tuple $y=\left(  y_{1}%
,y_{2},\ldots,y_{6}\right)  \in F^{6}$ satisfies%
\[
J_{u,v}\left(  y\right)  =J_{4,3}\left(  y\right)  =\left(  y_{4-i+j}\right)
_{1\leq i\leq3,\ 1\leq j\leq3}=\left(
\begin{array}
[c]{ccc}%
y_{4} & y_{5} & y_{6}\\
y_{3} & y_{4} & y_{5}\\
y_{2} & y_{3} & y_{4}%
\end{array}
\right)
\]
and thus $\det\left(  J_{u,v}\left(  y\right)  \right)  =y_{6}y_{3}^{2}%
-2y_{3}y_{4}y_{5}+y_{4}^{3}-y_{2}y_{6}y_{4}+y_{2}y_{5}^{2}$.
\end{example}

Why is Corollary \ref{cor.anzis} equivalent to \cite[Corollary 6.4]%
{Anzis-etc}? In fact, Corollary \ref{cor.anzis} can be restated in
probabilistic terms; then it says that a uniformly random $\left(
u+v-1\right)  $-tuple $y\in F^{u+v-1}$ satisfies $\det\left(  J_{u,v}\left(
y\right)  \right)  =0$ with a probability of $\dfrac{q^{u+v-2}}{q^{u+v-1}%
}=\dfrac{1}{q}$. However, the matrix $J_{u,v}\left(  y\right)  $ in Corollary
\ref{cor.anzis} is precisely the Jacobi--Trudi matrix\footnote{We are using
the terminology of \cite{Anzis-etc} here.} corresponding to the
rectangle-shaped partition $\left(  u^{v}\right)  $, except that the entries
of $y$ have been substituted for the complete homogeneous symmetric functions
$h_{1},h_{2},\ldots,h_{u+v-1}$. The determinant $\det\left(  J_{u,v}\left(
y\right)  \right)  $ therefore is the image of the Schur function $s_{\left(
u^{v}\right)  }$ under this substitution. Thus, Corollary \ref{cor.anzis} says
that when a uniformly random $\left(  u+v-1\right)  $-tuple of elements of $F$
is substituted for $\left(  h_{1},h_{2},\ldots,h_{u+v-1}\right)  $, the Schur
function $s_{\left(  u^{v}\right)  }$ becomes $0$ with a probability of
$\dfrac{1}{q}$. This is precisely the claim of \cite[Corollary 6.4]{Anzis-etc}.

We shall now sketch (on an example) how Corollary \ref{cor.anzis} can be
derived from our Corollary \ref{cor.hankel.detk}:

\begin{proof}
[Proof of Corollary \ref{cor.anzis} (sketched).]For a sufficiently
representative example, we pick the case when $u=2$ and $v=5$; the reader will
not find any difficulty in generalizing our reasoning to the general case.

Thus, we must show that the number of all $6$-tuples $y\in F^{6}$ satisfying
$\det\left(  J_{2,5}\left(  y\right)  \right)  =0$ is $q^{5}$. Let $y=\left(
y_{1},y_{2},\ldots,y_{6}\right)  \in F^{6}$ be any $6$-tuple. Then,%
\[
J_{2,5}\left(  y\right)  =\left(
\begin{array}
[c]{ccccc}%
y_{2} & y_{3} & y_{4} & y_{5} & y_{6}\\
y_{1} & y_{2} & y_{3} & y_{4} & y_{5}\\
y_{0} & y_{1} & y_{2} & y_{3} & y_{4}\\
y_{-1} & y_{0} & y_{1} & y_{2} & y_{3}\\
y_{-2} & y_{-1} & y_{0} & y_{1} & y_{2}%
\end{array}
\right)  =\left(
\begin{array}
[c]{ccccc}%
y_{2} & y_{3} & y_{4} & y_{5} & y_{6}\\
y_{1} & y_{2} & y_{3} & y_{4} & y_{5}\\
1 & y_{1} & y_{2} & y_{3} & y_{4}\\
0 & 1 & y_{1} & y_{2} & y_{3}\\
0 & 0 & 1 & y_{1} & y_{2}%
\end{array}
\right)
\]
(since $y_{0}=1$ and $y_{-1}=0$ and $y_{-2}=0$). If we turn the matrix
$J_{2,5}\left(  y\right)  $ upside down (i.e., we reverse the order of its
rows), then we obtain the matrix%
\[
\left(
\begin{array}
[c]{ccccc}%
0 & 0 & 1 & y_{1} & y_{2}\\
0 & 1 & y_{1} & y_{2} & y_{3}\\
1 & y_{1} & y_{2} & y_{3} & y_{4}\\
y_{1} & y_{2} & y_{3} & y_{4} & y_{5}\\
y_{2} & y_{3} & y_{4} & y_{5} & y_{6}%
\end{array}
\right)  ,
\]
which is precisely the Hankel matrix $H_{4,4}\left(  x\right)  $ for the
$9$-tuple%
\[
x=\left(  0,0,1,y_{1},y_{2},y_{3},y_{4},y_{5},y_{6}\right)  .
\]
Hence, this $9$-tuple $x$ satisfies $\det\left(  H_{4,4}\left(  x\right)
\right)  =\pm\det\left(  J_{2,5}\left(  y\right)  \right)  $ (since the
determinant of a matrix is multiplied by $\pm1$ when the rows of the matrix
are permuted). Therefore, the condition \textquotedblleft$\det\left(
J_{2,5}\left(  y\right)  \right)  =0$\textquotedblright\ is equivalent to the
condition \textquotedblleft$\det\left(  H_{4,4}\left(  x\right)  \right)
=0$\textquotedblright\ for this $9$-tuple $x$. Hence, the number of all
$6$-tuples $y\in F^{6}$ satisfying $\det\left(  J_{2,5}\left(  y\right)
\right)  =0$ is precisely the number of all $9$-tuples $x\in F^{9}$ that start
with the entries $0,0,1$ and satisfy $\det\left(  H_{4,4}\left(  x\right)
\right)  =0$. In other words, it is precisely the number of all $9$-tuples
$x\in F^{9}$ satisfying $x_{\left[  0,3\right)  }=\left(  0,0,1\right)  $ and
$\det\left(  H_{4,4}\left(  x\right)  \right)  =0$. However, Corollary
\ref{cor.hankel.detk} (applied to $k=3$ and $n=4$ and $a=\left(  0,0,1\right)
$) shows that the latter number is $q^{2\cdot4-3}=q^{5}$. This is precisely
what we wanted to show. Thus, Corollary \ref{cor.anzis} is proved.
\end{proof}

\end{document}